\documentclass[a4paper, reqno, 12pt]{amsart}

\usepackage[utf8]{inputenc}
\usepackage{amsthm,amsfonts,amssymb,amsmath,amsxtra,amsrefs}
\usepackage{mathtools}
 \usepackage{bbold}
\usepackage{latexsym}
 \usepackage{stmaryrd}
\usepackage{graphicx}
\usepackage{tikz}
\usetikzlibrary{arrows.meta}
\usepackage{tikz-cd}
\usepackage{todonotes,cancel}
\usepackage[margin=1.25in]{geometry}
\usepackage{mathrsfs}
\usepackage{bm}

\usepackage[all]{xy}
\SelectTips{cm}{}
\usepackage{xr-hyper}
\usepackage[colorlinks=false,
   citecolor=Black,
   linkcolor=Red,
   urlcolor=Blue]{hyperref}
\usepackage{verbatim}
\RequirePackage{xspace}
\RequirePackage{etoolbox}
\RequirePackage{varwidth}
\RequirePackage{enumitem}
\RequirePackage{tensor}
\RequirePackage{mathtools}
\RequirePackage{longtable}
\RequirePackage{multirow}

\usepackage{rotating}

\tikzset{
    labl/.style={anchor=south, rotate=90, inner sep=.5mm},
    satake double arrow/.style={
        double,
        double distance=1.7pt,
        -{Latex[length=2.8pt,width=4.2pt]},
        shorten <=1.2pt,
        shorten >=1.2pt
    }
}

\newtheorem{thm}{Theorem}
\newtheorem{theorem}[thm]{Theorem}
\newtheorem{thmintro}{Theorem}

\newtheorem{prop}[thm]{Proposition}
\newtheorem{proposition}[thm]{Proposition}
\newtheorem{lem}[thm]{Lemma}
\newtheorem{lemma}[thm]{Lemma}

\theoremstyle{definition}
\newtheorem{definition}[thm]{Definition}
\newtheorem{defi}[thm]{Definition}
\newtheorem{example}[thm]{Example}

\newtheorem{remark}[thm]{Remark}

\numberwithin{equation}{section}
\numberwithin{thm}{section}

\newcommand{\BC}{\ensuremath{\mathbb {C}}\xspace}

\newcommand{{\BG}}{\ensuremath{\mathbb {G}}\xspace}

\newcommand{{\BK}}{\ensuremath{\mathbb {K}}\xspace}

\newcommand{\BR}{\ensuremath{\mathbb {R}}\xspace}

\newcommand{\CB}{\ensuremath{\mathcal {B}}\xspace}

\DeclareMathOperator{\Int}{Int}
\begin{document}

\title[Total positivity and symmetric spaces]{Total Positivity and Symmetric Spaces}
\author[Huanchen Bao]{Huanchen Bao}
\address{Department of Mathematics, National University of Singapore, Singapore.}
\email{huanchen@nus.edu.sg}
\begin{abstract}
We define a notion of total positivity for the symmetric space $G/K$ by taking the Hausdorff closure of the image of Lusztig's totally positive part $G_{>0}$ in $G/K$. We introduce double Bruhat cells for the symmetric space and define their totally positive pieces. We prove a cell decomposition of the totally nonnegative symmetric space, give explicit positive parametrizations of all cells, establish closure relations, and show that the transition maps between the two natural families of parametrizations are subtraction-free. 
\end{abstract}
\maketitle
\tableofcontents

\section{Introduction}

\subsection{} Total positivity began with the classical study of matrices whose minors are all nonnegative or positive. Lusztig \cite{LuTotalPositivity} extended this theory from matrices to arbitrary split connected reductive groups, using deep positivity properties coming from his theory of canonical bases. One of the key features of Lusztig's theory is that total positivity is not merely a subset of real points; it comes with a rich cell structure. For $G$ itself, the relevant cells are closely related to double Bruhat cells. Lusztig also defined totally nonnegative flag varieties and partial flag varieties. Marsh and Rietsch \cite{MarshRietschParametrizationsFlagVarieties} gave explicit positive parametrizations of the cells in the full flag variety using positive subexpressions and generalized Chamber Ansatz coordinates, while Rietsch \cite{RietschClosureRelationsPartialFlag} proved closure relations for totally nonnegative cells in partial flag varieties.  The resulting theory of total positivity has become a central object connecting algebraic groups, canonical bases, cluster structures, Poisson geometry, and combinatorics.

\subsection{} The purpose of this paper is to develop the theory of total positivity for symmetric spaces. 

Let $G$ be a split connected reductive group over $\mathbb{R}$, let $\theta$ be an involution of $G$, and let $K=G^\theta$ be the fixed-point subgroup. We study the symmetric space
\[
\mathcal S=G/K.
\]
In the Lusztig theory of total positivity, one needs to fix a pinning for the reductive group $G$, including Borel subgroups $B^\pm$. We further require $(B^+, T)$ to be a $\theta$-split pair in the sense of Springer \cite{SpringerAlgebraicGroupsWithInvolutions}. In particular, $\theta(B^+) = \dot{w}^\bullet B^+  \dot{w}^{\bullet,-1} $ for some Weyl group element $w^\bullet$.  The involution $\theta$ is further normalized using the $\imath$-root datum attached to the symmetric pair; the precise convention is recalled in \S\ref{subsec:algebraic-groups}.  One should take this as the pinning data for the symmetric space. 

Throughout this paper we restrict to the class of symmetric pairs for which this normalized involution is compatible with the diagram involution, equivalently the parameters in \S\ref{subsec:algebraic-groups} may be chosen so that $\overline\zeta_i=-1$. This condition excludes a few cases (see \S\ref{subsec:algebraic-groups} for the detailed
list), but it is the natural setting in which the involution is compatible with the positivity structure used below.

\subsection{} The central geometric idea of the paper is to introduce the analogue of double Bruhat cells for the symmetric space. We use the embedding
\[
\iota:G/K\longrightarrow G,\qquad gK\longmapsto g\theta(g)^{-1},
\]
which realizes $G/K$ as a closed subvariety of $G$. Then, for $u\leq w^\bullet$ in the Weyl group of $G$, we define
\[
\mathring {\mathcal S}_u
=
\iota^{-1}(B^+\dot u B^-\dot w^\bullet)= \iota^{-1}(B^+\dot u \dot w^\bullet \theta(B^-)).
\]
We call the strata  $\mathring {\mathcal S}_u$ the double Bruhat cells of the symmetric space. Note that $\mathring {\mathcal S}_u = \emptyset$ if $u \not \le w^\bullet$.

This terminology is not only formal. These strata are motivated by Poisson geometry studied by Evens and Lu \cite{EvensLu2006}: they are the $T$-leaves for the natural Poisson structure on the symmetric space. Thus they play for $G/K$ the role played by double Bruhat cells in $G$ and by projected Richardson strata in partial flag varieties. In the special case where the symmetric space is the group itself, this construction recovers the usual double Bruhat cells in $G$. This Poisson-geometric interpretation is one of the main reasons these strata are the correct pieces on which to study total positivity. The relevant Frobenius splitting  and cluster structure on the double Bruhat cell shall appear in forthcoming works. 

We now define the totally nonnegative symmetric space. Let
\[
\mathcal S_{\geq0}=(G/K)_{\geq0}:=\overline{G_{>0}K/K},
\]
where the closure is taken in the Hausdorff topology on the real locus of $G/K$. For each $u\leq w^\bullet$, define
\[
\mathcal S_{u}^{>0}:=\mathcal S_{\geq0}\cap\mathring {\mathcal S}_u,
\qquad
\mathcal S_{u}^{\geq0}:=\overline{\mathcal S_{u}^{>0}}.
\]
We also define the totally positive symmetric space by
\[
\mathcal S_{>0}:=\mathcal S_{e}^{>0}.
\]

The indexing set for these cells can be described in two equivalent ways. From the symmetric space side, the strata are indexed by
\[
Q=\{u\in W\mid u\leq w^\bullet\}.
\]
The second indexing set is motivated by partial flag varieties and the projected Richardson strata that appear there:
\[
P=\{(v,w)\mid v\in W_{I_\bullet},\ w\in W^{I_\bullet},\ v\leq w\}.
\]
The two descriptions are related by the map
\[
\phi:P\longrightarrow Q,
\qquad
(v,w)\longmapsto wv^{-1}w^\bullet.
\]
The main combinatorial result of the paper is that this map is an isomorphism of posets; see Proposition~\ref{prop:poset-isomorphism}. Thus the symmetric space stratification is controlled by the same combinatorics that governs projected Richardson strata. The relevant involutions are
\[
\rho:Q\longrightarrow Q,\qquad u\longmapsto w^\bullet\theta(u^{-1})w^\bullet,
\qquad
\sigma:=\phi^{-1}\circ\rho\circ\phi:P\longrightarrow P.
\]
The anti-involution $x\mapsto\theta(x^{-1})$ fixes the embedded symmetric space $\iota(G/K)$ pointwise. Together with the Chevalley involution, it gives the stratum-label involution $\rho$; the conjugation by $w^\bullet$ accounts for the shift in the definition of $\mathring{\mathcal S}_u$. This relation is explained in \S\ref{sec:two-posets}. 

\subsection{}
We can now state the main results. Let $(v,w)\in P$, choose a reduced expression $\mathbf w$ of $w$, and let $\mathbf v_+$ be the positive subexpression of $v$ in $\mathbf w$ in terms of Marsh and Rietsch \cite{MarshRietschParametrizationsFlagVarieties}. Let $G^{>0}_{\mathbf v_+,\mathbf w}$ denote the corresponding Marsh--Rietsch positive parameter set. We further define the positive set $H^{>0}_{\mathbf v_+,\mathbf w}$ as the positive-root analogue of the Marsh--Rietsch set $G^{>0}_{\mathbf v_+,\mathbf w}$; it is defined precisely in \S2.4 by applying the Chevalley involution to the Marsh--Rietsch parameter set. Let $\overline{T}=T/T^\theta$ be the quotient torus
and let $\overline{T}_{>0}$ be its totally positive part. 

\begin{thmintro}[Theorem~\ref{thm:main-positive-strata}]
\begin{enumerate}

\item The space $\mathcal S_{\geq0}$ is entirely contained in both the open $B^+$-orbit and the open $B^-$-orbit of $G/K$, that is, 
\[
 \mathcal S_{\geq0} \subset (B^-K/K) \cap (B^+K/K).
\]
\item  For every $(v,w)\in P$, there is an isomorphism of semi-algebraic varieties 
\[
G^{>0}_{\mathbf v_+,\mathbf w}\times\overline{T}_{>0}
\xrightarrow{\sim}
\mathcal S_{\phi(v,w)}^{>0}.
\]
\item  For every $(v,w)\in P$, there is an isomorphism of semi-algebraic varieties 
\[
H^{>0}_{\mathbf v_+,\mathbf w}\times\overline{T}_{>0}
\xrightarrow{\sim}
\mathcal S_{\phi(\sigma(v,w))}^{>0}.
\]
\item The space $\mathcal S_{\geq0}$ admits the cell decomposition
\[
\mathcal S_{\geq0}
=
\bigsqcup_{u\in Q}\mathcal S_{u}^{>0}.
\]
\item We have the closure relations
\[
\overline{\mathcal S_{u}^{>0}}
=
\bigsqcup_{u\leq u'}\mathcal S_{u'}^{>0}.
\]
\item Let $u \in Q$. The cell $\mathcal S_{u}^{>0}$ is a connected component of $\mathring{\mathcal S}_u( \mathbb R)$.
\item Let $(v,w)\in P$ and write $\sigma(v,w)=(v',w')$.  After identifying the quotient torus $\overline{T}_{>0}$ with positive real coordinates, the transition map
\[
\mathbb{R}_{>0}^{\dim \overline{T} + \ell(w) - \ell(v)} \rightarrow G^{>0}_{\mathbf v_+,\mathbf w}\times\overline{T}_{>0}
\rightarrow
\mathcal S_{\phi(v,w)}^{>0}
\rightarrow
H^{>0}_{\mathbf v'_+,\mathbf w'}\times\overline{T}_{>0} \rightarrow \mathbb{R}_{>0}^{\dim \overline{T} + \ell(w) - \ell(v)}
\]
is subtraction-free.
 \item The total nonnegative symmetric space $\mathcal S_{\ge 0}$ is contractible.  
\end{enumerate} 
\end{thmintro}

\subsection{}
Lusztig \cite{LusztigTotalPositivitySymmetricSpaces} has already introduced a notion of total positivity for symmetric spaces. Our approach is different. First, our pinning is based on a split pair $(B^+, T)$, which is opposite to Lusztig's choice of a fundamental pair. This leads to different definitions of totally nonnegative symmetric spaces. Second, we use an entirely different stratification: the double Bruhat cell stratification motivated by Poisson geometry. This leads to a particularly simple cell decomposition, closure order, and subtraction-free transition functions. In particular, after passing to the quotient torus, the transition functions are subtraction-free without taking square roots. 

\subsection{}
The paper is organized as follows. In \S2 we recall the necessary background on symmetric spaces, the embedding $G/K\hookrightarrow G$, Lusztig total positivity, and Marsh--Rietsch parametrizations. In \S3.1 we prove the poset isomorphism between the symmetric space indexing set and the projected Richardson indexing set. In \S3.2 we define the totally nonnegative symmetric space and state the main theorem. In \S3.3 we carry out the rank-one calculations. In \S3.4 we introduce a special regular function used to prove that $\mathcal S_{\geq0}$ lies in the open $B^\pm K/K$ charts. This also leads to a new proof of the core geometric structure of $G_{\ge 0}$ bypassing Lusztig's theory of canonical bases.  Finally, \S3.5 proves the main theorem.

\vspace{.2cm}
\noindent {\bf Acknowledgment: } HB is supported by the MOE grant A-8003582-00-00.  
\section{Preliminaries}

\subsection{Algebraic groups, involutions, $\imath$-root data}
\label{subsec:algebraic-groups}

Let $G$ be a connected reductive algebraic group over $\BC$. We assume throughout that $G$ splits over $\BR$. Fix a split maximal torus $T\subset G$ and a Borel subgroup $B^+\subset G$ containing $T$. Let $B^-$ be the opposite Borel subgroup with respect to $T$. Write $X=X^\ast(T)$ and $Y=X_\ast(T)$ for the character and cocharacter lattices, equipped with their natural perfect pairing $\langle \cdot, \cdot \rangle$. Let $\Phi\subset X$ be the root system, let $\Phi^\vee\subset Y$ be the coroot system, and let $\Delta=\{\alpha_i\}_{i\in I}$ be the simple roots determined by $B^+$. We write $W$ for the Weyl group.

For later use we fix a pinning. For each simple root $\alpha_i$, let $x_i:\BG_a\to G$ and $y_i:\BG_a\to G$ denote the corresponding positive and negative simple root subgroup homomorphisms, and let $\alpha_i^\vee:\BG_m\to T$ be the simple coroot. We set
\[
\dot s_i=x_i(-1)y_i(1)x_i(-1)\in N_G(T),
\]
so that the image of $\dot s_i$ in $W$ is the simple reflection $s_i$. For any $w\in W$, after choosing a reduced expression $w=s_{i_1}\cdots s_{i_r}$, we write $\dot w=\dot s_{i_1}\cdots \dot s_{i_r}\in N_G(T)$. This is independent of the reduced expression. 

We further fix an $\imath$-root datum, in the sense of \cite{SongSymmetricSubgroupSchemes}*{\S 2.4.1} and \cite{BaoSongSymmetricSubgroupSchemesFrobenius}*{\S 2.4.2},
\[
(Y,X,\{\alpha_i^\vee\}_{i\in I},\{\alpha_i\}_{i\in I},\theta)  \text{ of type $(I=I_\bullet\sqcup I_\circ,\tau)$.}
\]
The induced involution on the weight lattice is characterized by
\[
\theta(\alpha_i)=-w_\bullet(\alpha_{\tau(i)})\qquad (i\in I).
\]

We denote by $L_{I_\bullet}$ the standard Levi subgroup corresponding to $I_\bullet$. Its Weyl group is $W_{I_\bullet}$, and $w_\bullet\in W_{I_\bullet}$ is its longest element. We also write $w_0\in W$ for the longest element of the full Weyl group.   Let $W^{I_\bullet}\subset W$ be the set of minimal-length representatives of $W/W_{I_\bullet}$. We write $w^\bullet = w_0 w_\bullet$. By definition of the $\imath$-root datum \cite{SongSymmetricSubgroupSchemes}*{\S 2.4.1} and \cite{BaoSongSymmetricSubgroupSchemesFrobenius}*{\S 2.4.2},  we have $w_0 w_\bullet = w_\bullet w_0$. Hence  $w^\bullet = w^{\bullet, -1}$.

We record automorphisms that will appear repeatedly. First, we write $\omega$ for the Chevalley involution attached to the chosen pinning; it is characterized by
\[
\omega(x_i(a))=y_i(a),\qquad \omega(y_i(a))=x_i(a),\qquad \omega(t)=t^{-1}\quad (t\in T).
\]
Second, after choosing representatives $\dot w_0,\dot w_\bullet\in N_G(T)$, we have $\Int(\dot w_0):G\to G$, $g\mapsto \dot w_0g\dot w_0^{-1}$, and $\Int(\dot w_\bullet):G\to G$, $g\mapsto \dot w_\bullet g\dot w_\bullet^{-1}$. 
Following Song's functorial construction of involutions on $G$ in \cite{SongSymmetricSubgroupSchemes}*{\S 3.3}, one introduces parameters $\overline\zeta_i \in \{\pm1\}$ such that
\[
\overline\zeta_i\overline\zeta_{\tau(i)}=(-1)^{\langle 2\rho_\bullet^\vee,\alpha_i\rangle}.
\]
Here $\rho_\bullet^\vee$ denotes the half-sum of the positive coroots in the root subsystem generated by $I_\bullet$.
In particular, we define the involution on $G$
\[
\theta=\Int(\dot w_\bullet)\circ\tau\circ\omega\circ\Xi(\overline\zeta),
\]
where $\Xi(\overline\zeta)$ denotes the torus automorphism attached to the parameters $\overline\zeta_i$. We use the same symbol $\theta$ for this group involution and for the induced involutions on $X$, $Y$, and $W$.
With this normalization, the involution $\theta$ is the identity on the derived subgroup of the Levi $L_{I_\bullet}$. One has the explicit formulas
\[
\theta\bigl(y_i(a)\bigr)=\dot w^{-1}_\bullet\,x_{\tau(i)}((-1)^{\langle 2 \rho_\bullet^\vee, \alpha_{\tau(i)} \rangle}  \overline{\zeta}_i^{-1} a) \,\dot w_\bullet= \dot w^{-1}_\bullet\,x_{\tau(i)}(   \overline{\zeta}_{\tau(i)} a)\,\dot w_\bullet,
\qquad i\in I_\circ,
\]
\[
\theta\bigl(x_i(a)\bigr)=\dot w_\bullet\,y_{\tau(i)}(\overline{\zeta}_i a)\,\dot w_\bullet^{-1},
\qquad i\in I_\circ.
\]
Applying the same formulas to the node $\tau(i)$, one also has
\[
\theta\bigl(y_{\tau(i)}(a)\bigr)
=\dot w_\bullet^{-1}\,x_i\left((-1)^{\langle 2\rho_\bullet^\vee,\alpha_i\rangle}\overline\zeta_{\tau(i)}^{-1}a\right)\,\dot w_\bullet
=\dot w_\bullet^{-1}\,x_i(\overline\zeta_i a)\,\dot w_\bullet,
\qquad i\in I_\circ,
\]
and
\[
\theta\bigl(x_{\tau(i)}(a)\bigr)
=\dot w_\bullet\,y_i(\overline\zeta_{\tau(i)}a)\,\dot w_\bullet^{-1},
\qquad i\in I_\circ.
\]

Let $K=G^\theta$ be the fixed-point subgroup. Let ${\mathcal  S}=G/K$ be the symmetric space. Following \cite{BaoSongCoordinateRingsSymmetricSpaces}, the variety ${\mathcal  S}$ is defined over $\mathbb Z$. The involution induces an involution on $W$, still denoted by $\theta$, such that $\theta(w) = w_\bullet \tau(w) w_\bullet^{-1}$.

 It follows from the construction that $(B^+,T)$ is a $\theta$-split pair in the sense of Springer \cite{SpringerAlgebraicGroupsWithInvolutions}. In particular, $B^\pm K/K$ is the open $B^\pm$-orbit in $\mathcal{S}$, and we have $\theta(B^\pm)=\dot w^\bullet  B^\pm \dot w^{\bullet,-1}$.

\vspace{.5cm}

\textbf{From now on, we assume $(-1)^{\langle 2\rho_\bullet^\vee,\alpha_i\rangle} = 1$  for all $i \in I$.} We then take $\overline\zeta_i=-1$ for every $i\in I$ in the definition of $\theta$.

As a consequence, the involution $\theta$ commutes with both $\omega$ and the diagram involution $\tau$. The parity assumption is a local condition on the Satake diagram: for a white node $i$, only the black connected components adjacent to $i$ contribute to $\langle 2\rho_\bullet^\vee,\alpha_i\rangle$. Thus it can be checked on the real-rank-one subdiagrams attached to the $\tau$-orbits of white nodes; see \S\ref{subsec:Rank-One} and \cite{BaoWangiCanonicalBasis}*{\S3.2, Definition~3.2 and Tables~1--2}. At a black node the pairing is $2$, so the condition is automatic. Among the real-rank-one diagrams, the condition fails precisely for AIV$_m$ with $m$ odd.

\begin{remark}
The choice of $\overline\zeta_i=-1$ rather than $+1$ is a matter of normalization/pinning. The essential condition is instead $\zeta_i = \zeta_{\tau(i)} \in \mathbb{R}$; this is the actual obstruction in the excluded cases. 
\end{remark}
\subsection{The embedding map}\label{sec:Su}

Define the map $\iota:G\to G$ by $\iota(g)=g\theta(g)^{-1}$.
Since $\theta(k)=k$ for every $k\in K$, one has $\iota(gk)=g\theta(g)^{-1}$. So $\iota$ is constant on right $K$-cosets. Therefore it descends to a morphism
\[
\iota:G/K\longrightarrow G.
\]

The twisted $G$-action on $G$ is given by $g\ast x=gx\theta(g)^{-1}$ for $g\in G$ and $x\in G$. Let $S_\theta=\{g\theta(g)^{-1}\mid g\in G\}\subset G$ be the twisted orbit of $e \in G$.
By Springer \cite{SpringerAlgebraicGroupsWithInvolutions}*{Proposition~2.2}, $S_\theta$ is a closed subvariety of $G$, and the induced map
\[
G/K\longrightarrow S_\theta,\qquad gK\longmapsto g\theta(g)^{-1}
\]
is an isomorphism of affine $G$-varieties. The $B^\pm$-orbits on $G/K$ may equivalently be viewed as the corresponding twisted $B^\pm$-orbits on $S_\theta$. By Springer and Richardson--Springer, there are only finitely many such orbits \cites{SpringerAlgebraicGroupsWithInvolutions,RichardsonSpringerBruhatOrderSymmetricVarieties}.

Let $\overline{T} = T/T^\theta$ be the quotient torus with the character lattice $\breve{X}= \{\mu - \theta(\mu) \vert \mu \in X\}$. Let $\breve{Y}$ be the cocharacter lattice of $\overline{T}$. Let $r = \dim \overline{T}$. By the local structure theorem \cite{DP}*{Proposition~3.8}, we have  $B^-K/K \cong U_{P_{I_\bullet}^-} \times \overline{T}$. In particular, $\dim \mathcal S = \ell(w^\bullet) +r$.

We next recall the stratification of $  {\mathcal  S}$ motivated from the Poisson structure. For $u\in W$, define
\[
\mathring {\mathcal  S}_u=\iota^{-1}(B^+\dot uB^-\dot w_0\dot w_\bullet)=\iota^{-1}(B^+\dot u\dot w_0\dot w_\bullet\theta(B^-)).
\]
We write ${\mathcal  S}_u=\overline{\mathring {\mathcal  S}_u}$ for the Zariski closure. 
\begin{definition}
We call the stratum $\mathring {\mathcal S}_u$ the (open) double Bruhat cell for the symmetric space $G/K$. 
\end{definition}
In the group case, equivalently for the symmetric space $(G\times G)/\Delta(G)$, this construction recovers the usual double Bruhat cells of $G$. This motivates the name.

\begin{remark}\label{rem:GG}
Let $\tilde{G}= G \times G$. We consider the involution \[
\tilde{\theta}: G \times G \rightarrow G \times G, \qquad (g,h) \mapsto (\theta(h), \theta(g)).
\]
Denote the fixed-point group by $\tilde{K} = \tilde{G}^{\tilde{\theta}}$. Then we consider the symmetric space $\tilde{G}/\tilde{K}$. It follows from \cite{EvensLu2006} that the intersections of diagonal $G$-orbits with $(B^+, B^-)$-orbits are precisely the $T$-leaves of the natural Poisson structure on $\tilde{G}/\tilde{K}$ (as a Poisson homogeneous space): see also \cite{Lu2014TLeaves}.   The map $G \times G \rightarrow G, (g,h) \mapsto g \theta(h)^{-1}$ induces an isomorphism $\tilde{G}/\tilde{K} \cong G$.  Under this isomorphism, diagonal $G$-orbits in $\tilde{G}/\tilde{K}$ are identified with the twisted conjugacy classes, and $(B^+, B^-)$-orbits are identified with $(B^+, \theta(B^-))$ double cosets in $G$.
Under this isomorphism $\mathcal S$ is precisely the diagonal $G$-orbit of $\tilde{K}$. This leads to the stratum $\mathring {\mathcal  S}_u$.
\end{remark}

\begin{lemma}\label{le:Su}
    \begin{enumerate}
    \item $\mathring {\mathcal  S}_u\neq\emptyset$ if and only if $u\leq w^\bullet$ in the Bruhat order.
    \item ${\mathcal  S}_u=\bigsqcup_{u\leq v\leq w^\bullet}\mathring {\mathcal  S}_v.$
    \item For $u\leq w^\bullet$, $\mathring {\mathcal  S}_u\subset {\mathcal  S}$ is a smooth and connected subvariety of dimension  $\ell(w^\bullet)- \ell(u) +r $.
    \end{enumerate}
\end{lemma}

\begin{proof}
We prove Part (1).  Let $\theta'=\Int(\dot w^\bullet)\theta$. Then $\theta'$ fixes $B^+$ and $B^-$. Let $\mathcal C$ be the $\theta'$-twisted conjugacy class of $(\dot w^\bullet)^{-1}$. Then $\iota(G/K)=\mathcal C\cdot(\dot w^\bullet)^{-1}$. Since $B^+K/K\subset G/K$ is dense, $B^+  (\dot w^\bullet)^{-1} B^+\cap\mathcal C$ is dense in $\mathcal C$. By \cite{CLT10}*{\S 2.4},
\[
\mathcal C\cap B^+\dot uB^-=(\iota(G/K)\cap B^+\dot uB^-\dot w^\bullet)(\dot w^{\bullet})^{-1}\neq\emptyset
\]
if and only if $u\leq w^\bullet = w^{\bullet,-1}$.

We prove Part (2). We follow the notations in Remark~\ref{rem:GG}.  
By the Bruhat decomposition, we have 
\[
\overline{B^+\dot uB^-\dot w^\bullet}=\bigsqcup_{u\leq v}B^+\dot vB^-\dot w^\bullet.
\]
We can equivalently consider the closure relations in $\tilde{G}/\tilde{K}$, since $\mathcal S$ is closed in $G$. Then we can apply results in \cite{Ri92}. In particular, by \cite{Ri92}*{Theorem~2.1}, we have
\[
{\mathcal S}_u=\overline{{\mathcal S}\cap\iota^{-1}(B^+\dot uB^-\dot w^\bullet)}={\mathcal S} \cap\iota^{-1}(\overline{B^+\dot uB^-\dot w^\bullet})=\bigsqcup_{u\leq v\leq   w^\bullet}\mathring {{\mathcal S}}_v.
\]
The finiteness condition in \cite{Ri92}*{Theorem~2.1} is not needed here, since ${\mathcal S}$ is closed. Finally, Part (3) follows from \cite{Ri92}*{Corollary~1.5}. 
\end{proof}

\subsection{Lusztig's total positivity}\label{sec:Lutp}

We recall Lusztig's theory of total positivity for a split reductive group with the fixed pinning above  \cite{LuTotalPositivity}*{\S\S 2.2--2.12}. Let $T_{>0}:=\langle \chi(a)\mid \chi \in Y,\ a\in\mathbb R_{>0}\rangle\subset T(\mathbb R)$ be the submonoid (with $1$) generated by $\chi(a)$ for various $\chi \in Y,\ a\in\mathbb R_{>0}$.
For a reduced expression $w=s_{i_1}\cdots s_{i_r}$, set
\[
U^+_{w,>0}=\{x_{i_1}(a_1)\cdots x_{i_r}(a_r)\mid a_1,\ldots,a_r \in \mathbb R_{>0}\}
\]
and
\[
U^-_{w,>0}=\{y_{i_1}(a_1)\cdots y_{i_r}(a_r)\mid a_1,\ldots,a_r\in \mathbb R_{>0}\}.
\]
These subsets are independent of the chosen reduced expression. We write $U^+_{>0}:=U^+_{w_0,>0}$ and $U^-_{>0}:=U^-_{w_0,>0}$. We also write $B^-_{>0} =U^-_{>0} T_{>0} $ and $B^+_{>0} =U^+_{>0} T_{>0}$.

The totally positive part of $G$ is
\[
G_{>0}:=U^+_{>0}T_{>0}U^-_{>0}=U^-_{>0}T_{>0}U^+_{>0}, 
\]
and the totally nonnegative part $G_{\geq0}$ is the submonoid of $G(\mathbb R)$ generated by $T_{>0}$ together with all $x_i(a)$ and $y_i(a)$ for $i\in I$ and $a\in\mathbb R_{\geq0}$.  

We shall frequently use the following elementary exchange relations associated with the pinning \cite{LuTotalPositivity}*{\S\S 1.1--1.3}. For $t\in T$, one has
\[
tx_i(a)=x_i(\alpha_i(t)a)t,\qquad ty_i(a)=y_i(\alpha_i(t)^{-1}a)t.
\]
For distinct simple roots $i\neq j$, the simple opposite root subgroups commute:
\[
y_i(a)x_j(c)=x_j(c)y_i(a).
\]
Finally, in rank one one has the relation
\begin{equation}\label{eq:Luxy}
y_i(a)\alpha_i^\vee(b^{-1})x_i(c)=x_i\!\left(\frac{c}{ac+b^2}\right)\alpha_i^\vee\!\left(\frac{b}{ac+b^2}\right)y_i\!\left(\frac{a}{ac+b^2}\right),
\end{equation}
whenever the displayed expressions are defined. In particular, if $a,b,c>0$, then $ac+b^2>0$, and all three transformed parameters are again positive. 



Recall the Demazure product $\ast$ \cite{BaoHe2022TwistedProduct}*{\S2.6} on the set $W$, determined by the following two rules: 
\begin{itemize}
\item $x \ast y = xy$ if  $x, y \in W$   such that $\ell(xy) = \ell(x) + \ell(y)$;
\item $s_i \ast w = w$ if $i \in I$, $w \in W$ such that $ s_i w < w$.
\end{itemize} 

By direct calculation, the multiplication map $U_{s_i, >0}^- \times U_{s_i, >0}^- \rightarrow U_{s_i, >0}^-$ is surjective. We conclude that 

{\it (a) The multiplication map $U_{s, >0}^- \times U_{t, >0}^- \rightarrow U_{s \ast t, >0}^-$ is surjective for $s, t \in W$.}

\subsection{Total positivity on flag varieties}\label{sec:TPFlag}

We recall the totally nonnegative flag varieties and the Marsh--Rietsch parametrizations. Let $\mathcal B$ denote the flag variety of $G$, and identify $\mathcal B\cong G/B^+$. Its totally nonnegative part is
\[
\mathcal B_{\geq0}=(G/B^+)_{\geq0}:=\overline{G_{>0}B^+/B^+}.
\]
For $v,w\in W$ with $v\leq w$, one considers the Richardson stratum
\[
R_{v,w}=B^+\dot wB^+/B^+\cap B^-\dot vB^+/B^+.
\]
The totally positive Richardson stratum is
\[
R^{>0}_{v,w}:=R_{v,w}\cap(G/B^+)_{\geq0}.
\]
Marsh--Rietsch prove that $R^{>0}_{v,w}$ is a positive cell, explicitly parametrized by the positive Marsh--Rietsch coordinates \cite{MarshRietschParametrizationsFlagVarieties}*{\S 11}, which we shall recall. 

Fix a reduced expression $\mathbf w=s_{i_1}\ldots s_{i_n}$ for $w$, and let $\mathbf v_+=(v_{(0)},v_{(1)},\ldots,v_{(n)})$ be the positive subexpression of $v$ in $\mathbf w$ in the sense of \cite{MarshRietschParametrizationsFlagVarieties}*{Definition~3.4}; its existence and uniqueness are given in \cite{MarshRietschParametrizationsFlagVarieties}*{Lemma~3.5}. Write
\[
J^+_{\mathbf v_+}=\{k\mid v_{(k)}=v_{(k-1)}s_{i_k}\},\qquad J^\circ_{\mathbf v_+}=\{k\mid v_{(k)}=v_{(k-1)}\}.
\]
The Marsh--Rietsch positive set \cite{MarshRietschParametrizationsFlagVarieties}*{Definition~11.2} is
\[
G^{>0}_{\mathbf v_+,\mathbf w}=\left\{g=g_1\cdots g_n\ \middle|\ g_k=\begin{cases}\dot s_{i_k},&k\in J^+_{\mathbf v_+},\\ y_{i_k}(t_k),&k\in J^\circ_{\mathbf v_+},\ t_k>0.\end{cases}\right\} \subset U^- \dot{v} \cap B^+ \dot{w}B^+.
\]
By \cite{MarshRietschParametrizationsFlagVarieties}*{Theorem~11.3}, the map
\[
G^{>0}_{\mathbf v_+,\mathbf w}\longrightarrow R^{>0}_{v,w},\qquad g\longmapsto gB^+
\]
is an isomorphism of real semi-algebraic varieties.

We also need the corresponding positive-root version. Recall Chevalley involution $\omega$ from Section~\ref{subsec:algebraic-groups}.
It sends our chosen representative $\dot s_i$ to $\dot s_i^{-1}$. Therefore the positive-root set below is precisely the image of the Marsh--Rietsch positive set under $\omega$:
\[
H^{>0}_{\mathbf v_+,\mathbf w}:=\omega(G^{>0}_{\mathbf v_+,\mathbf w}) = \left\{h=h_1\cdots h_n\ \middle|\ h_k=\begin{cases}\dot s_{i_k}^{-1},&k\in J^+_{\mathbf v_+},\\ x_{i_k}(t_k),&k\in J^\circ_{\mathbf v_+},\ t_k>0.\end{cases}\right\}.
\]

The subset $H^{>0}_{\mathbf v_+,\mathbf w}$ gives the Marsh--Rietsch parametrization when one chooses $B^-$, rather than $B^+$, as the base point of $\CB$. Recall Lusztig's theorem  \cite{LuTotalPositivity}*{Theorem~8.7} that the totally nonnegative part of the flag variety is invariant under the involution $\omega$.

\subsection{Total positivity on partial flag varieties}\label{sec:partial-flag-tnn}
We now pass from the full flag variety to a partial flag variety. Let $P^+_{I_\bullet}\supset B^+$ be the standard parabolic subgroup corresponding to $I_\bullet$, and let $\mathcal P_{I_\bullet}=G/P^+_{I_\bullet}$ be the corresponding partial flag variety. Denote by $\pi_{I_\bullet}:G/B^+\to G/P^+_{I_\bullet}$ the natural projection. The totally nonnegative part of the partial flag variety is defined by
\[
(\mathcal P_{I_\bullet})_{\geq0}:=\pi_{I_\bullet}((G/B^+)_{\geq0}).
\]
For $(v,w)\in W_{I_\bullet}\times W^{I_\bullet}$ with $v\leq w$, we write
\[
\Pi^{>0}_{v,w}:=\pi_{I_\bullet}(R^{>0}_{v,w})\subset(\mathcal P_{I_\bullet})_{\geq0}.
\]
These are the totally nonnegative projected Richardson varieties; the corresponding cell decompositions and closure relations for partial flag varieties are studied by Rietsch \cite{RietschClosureRelationsPartialFlag}. 

Let $P=\{(v,w) \vert v\in W_{I_\bullet}, w\in W^{I_\bullet}, v\le w\}$. We equip $P$ with the partial order $\preceq$ defined by declaring $(v,w)\preceq(v',w')$ if and only if there exists $z\in W_{I_\bullet}$ such that $v\le v'z\le w'z\le w$. The following claims can be found in \cite{RietschClosureRelationsPartialFlag}.

{\it (a) We have $(U^-P^+_{I_\bullet}/P^+_{I_\bullet})\cap(\mathcal P_{I_\bullet})_{\geq0}=\displaystyle\sqcup_{\substack{(v,w)\in P}}\Pi^{>0}_{v,w}$.}

{\it (b) For $(v,w) \in P$, we have $\overline{\Pi^{>0}_{v,w}} \cap (U^-P^+_{I_\bullet}/P^+_{I_\bullet}) = \sqcup_{(v',w') \in P, (v,w) \preceq (v',w') } \Pi^{>0}_{v',w'}$.}

{\it (c) For $(v,w) \in P$, the map $G^{>0}_{\mathbf v_+,\mathbf w}\longrightarrow \Pi^{>0}_{v,w}$, $g\longmapsto gP^+_{I_\bullet}$
is an isomorphism of real semi-algebraic varieties.}

Recall that $U^-P^+_{I_\bullet}/P^+_{I_\bullet} \cong U_{P^-_{I_{\bullet}}}$, where $P^-_{I_{\bullet}} = \omega(P^+_{I_{\bullet}})$ and $U_{P^-_{I_{\bullet}}}$ denotes the unipotent radical of $P^-_{I_{\bullet}}$. So one could interpret the claims as the total positivity structure on $U_{P^-_{I_{\bullet}}}$.

\begin{lem}\label{lem:contractible}
The space  $(U^-P^+_{I_\bullet}/P^+_{I_\bullet})\cap(\mathcal P_{I_\bullet})_{\geq0}$ is contractible.
\end{lem}

\begin{proof}
By (the proof of) \cite{BaoHe2024Product}*{Proposition~5.6}, the space $(U^-P^+_{I_\bullet}/P^+_{I_\bullet})\cap(\mathcal P_{I_\bullet})_{\geq0}$ is a cone over the link of $(\mathcal P_{I_\bullet})_{\geq0}$ at $e P^+_{I_\bullet}$. Hence it is contractible by \cite{BaoHe2024Product}*{\S5.4.2}.

Alternatively,  we can apply the same argument as \cite{Lusztig1998}*{\S4}. For any $hP^+_{I_\bullet} \in (U^-P^+_{I_\bullet}/P^+_{I_\bullet})\cap(\mathcal P_{I_\bullet})_{\geq0}$ and any $g \in G_{>0}$, we have $ghP^+_{I_\bullet}  \in \Pi_{e, w^\bullet}^{>0}$ by direct computation using the Marsh-Rietsch parametrization. Therefore $(U^-P^+_{I_\bullet}/P^+_{I_\bullet})\cap(\mathcal P_{I_\bullet})_{\geq0}$ is homotopic to $\Pi_{e, w^\bullet}^{>0} \cong \mathbb{R}_{>0}^\ell(w^\bullet)$, and hence contractible.
\end{proof}

\section{Total Positivity in Symmetric Spaces}

\subsection{Two posets attached to the symmetric space}\label{sec:two-posets}

We recall the following proposition from \cite{BaoHe2021}*{Proposition~4.6}. Since $W_{I_\bullet}$ is finite, the statements are simplified by \cite{BaoHe2021}*{Remark~2.2}.
\begin{proposition}\label{prop:BH}
Let $x,x'\in W_{I_\bullet}$ and $y,y'\in W^{I_\bullet}$. The following conditions are equivalent:
\begin{enumerate}
\item $w_\bullet x'(y')^{-1}\leq w_\bullet xy^{-1}$;
\item There exists $u\in W_{I_\bullet}$ such that $x\leq x'u$ and $y'u\leq y$.
\end{enumerate}
\end{proposition}

Let $Q=\{u\in W\mid u\le w_0w_\bullet\}$, equipped with the Bruhat order inherited from $W$. Let $Q^{\mathrm{opp}}$ be the opposite poset of $Q$. Recall the poset $P$ in \S\ref{sec:partial-flag-tnn}.
\begin{prop}
\label{prop:poset-isomorphism}
The map
\[
\phi:P\longrightarrow Q,\qquad (v,w)\longmapsto wv^{-1}w^\bullet
\]
is an isomorphism of posets.
\end{prop}

\begin{proof}
Right multiplication by $w_0$ reverses the Bruhat order, while inversion preserves it. Since $w^\bullet=w_\bullet w_0$, we have
\[
\phi(v,w)\leq w^\bullet
\quad\Longleftrightarrow\quad
w_\bullet\leq wv^{-1}w_\bullet
\quad\Longleftrightarrow\quad
w_\bullet\leq w_\bullet vw^{-1}.
\]
By Proposition~\ref{prop:BH}, the last condition is equivalent to the existence of $z\in W_{I_\bullet}$ such that $v\leq z\leq w$, and hence to $v\leq w$. This proves that $\phi$ is well-defined.

For $u\in Q$, write the unique parabolic factorization $uw_0=wz$, where $w\in W^{I_\bullet}$ and $z\in W_{I_\bullet}$, and put $v=w_\bullet z^{-1}$. Then $u=wv^{-1}w^\bullet$, and the equivalence above shows that $v\leq w$. This gives the inverse of $\phi$ and proves bijectivity.

Finally, Proposition~\ref{prop:BH} gives
\[
\begin{aligned}
\phi(v,w)\leq\phi(v',w')
&\quad\Longleftrightarrow\quad
w_\bullet v'(w')^{-1}\leq w_\bullet vw^{-1}\\
&\quad\Longleftrightarrow\quad
\exists z\in W_{I_\bullet}:\ v\leq v'z\leq w'z\leq w\\
&\quad\Longleftrightarrow\quad
(v,w)\preceq(v',w').
\end{aligned}
\]
Here the middle inequality $v'z\leq w'z$ follows from $v'\leq w'$ and $w'\in W^{I_\bullet}$. Thus $\phi$ is an isomorphism of posets.
\end{proof}

Set
\begin{equation}\label{eq:rho-definition}
\rho:Q\longrightarrow Q,\qquad
\rho(u)=w^\bullet\theta(u^{-1})w^\bullet=w_0\tau(u^{-1})w_0.
\end{equation}
Inversion and the diagram automorphisms $\tau$ and $u\mapsto w_0uw_0$ preserve the Bruhat order and length. They commute, and $\rho(w^\bullet)=w^\bullet$. Thus $\rho$ is a length-preserving involution of the poset $Q$. Define
\begin{equation}\label{eq:sigma-definition}
\sigma:=\phi^{-1}\circ\rho\circ\phi:P\longrightarrow P.
\end{equation}
Explicitly,
\begin{equation}\label{eq:sigma-explicit}
\sigma(v,w)=(v',w')
\quad\Longleftrightarrow\quad
w'(v')^{-1}=\theta(wv^{-1})^{-1}=v\theta(w^{-1}).
\end{equation}
In particular, $\sigma$ is an involution of $P$, and
\[
\phi(\sigma(v,w))=v\theta(w^{-1})w^\bullet,
\qquad
\ell(w')-\ell(v')=\ell(w)-\ell(v).
\]
The length equality follows from
$\ell(\phi(v,w))=\ell(w^\bullet)-\ell(w)+\ell(v)$ and the length invariance of $\rho$.

\begin{remark}
We explain the geometric origin of these formulas. The map
\[
\alpha:G\longrightarrow G,\qquad \alpha(x)=\theta(x^{-1}),
\]
is an anti-involution, and for $x=g\theta(g)^{-1}$ one has $\alpha(x)=x$. Thus $\alpha$ fixes $\iota(G/K)$ pointwise. The Chevalley involution $\omega$ commutes with $\theta$, preserves $K$, and hence induces an involution on $G/K$. 
It follows by direct computation that
\begin{equation}\label{eq:omega-strata}
\omega(\mathring{\mathcal S}_u)=\mathring{\mathcal S}_{\rho(u)}.
\end{equation}
Thus $\rho$ is the action of $\omega$ on the stratum labels, computed using the symmetric-space identity $\alpha(x)=x$. 
\end{remark}
\subsection{Totally nonnegative symmetric spaces}
\begin{defi}
We define the totally nonnegative symmetric space by
\[
\mathcal S_{\geq0}:=\overline{G_{>0}K/K},
\]
where the closure is taken in the Hausdorff topology. 
\end{defi}
By definition, $\mathcal S_{\geq0}$ is contained in the real locus $\mathcal S (\mathbb{R})$ of the symmetric space. We also have the natural action of $G_{\ge 0}$ on $\mathcal S_{\ge0}$. We stratify this space by intersecting it with the double Bruhat cells $\mathring{\mathcal S}_u$ from \S\ref{sec:Su}. For $u\in Q$, we define
\[
\mathcal S_{u}^{>0}:=\mathcal S_{\geq0}\cap \mathring {\mathcal  S}_u,
\qquad
\mathcal S_{u}^{\geq0}:=\overline{\mathcal S_{u}^{>0}}.
\]
We also define the totally positive symmetric space to be
\[
\mathcal S_{>0}:=\mathcal S_{e}^{>0}.
\]

Recall the quotient torus $\overline{T} = T / T^\theta$. For any $t \in T$, we denote by $\overline{t}$ its image in $\overline{T}$. We have $\overline{T} \cong TK/K$. Note that the induced map $T(\mathbb{R}) \rightarrow \overline{T}(\mathbb{R})$ on $\mathbb {R}$-points is generally not surjective. We define $\overline{T}_{>0} = \langle \chi(a) \vert \chi \in \breve{Y}, a \in \mathbb R_{>0} \rangle$. 

\begin{lem}\label{lem:Tbar}
The quotient map $T_{>0} \rightarrow \overline{T}_{>0}$ is surjective. 
\end{lem}

\begin{proof}
The image of the cocharacter lattice $Y$ has finite index in $\breve{Y}$ via the quotient map. On the other hand, for any $a \in \mathbb R_{>0}$, the (positive) $n$-th root $\sqrt[n]{a}$ always exists in $\mathbb {R}_{>0}$. This proves the lemma.
\end{proof}

We now state the main theorem of this paper. 
\begin{theorem}\label{thm:main-positive-strata}
\begin{enumerate}

\item The totally nonnegative symmetric space satisfies
\[
\mathcal S_{\geq0}\subset (B^-K/K)\cap(B^+K/K).
\]

\item Let $(v,w)\in P$, and put $u=\phi(v,w)=wv^{-1}w^\bullet$. We have an isomorphism of semi-algebraic varieties
\[
G^{>0}_{\mathbf v_+,\mathbf w}\times  \overline{T}_{>0} \xrightarrow{\sim}\mathcal S^{>0}_{u}, \quad (g, \overline{t}) \mapsto gtK.
\]
\item Let $(v,w)\in P$, and put $u=\phi(\sigma(v,w))=v\theta(w^{-1})w^\bullet$. We have an isomorphism of semi-algebraic varieties
\[
H^{>0}_{\mathbf v_+,\mathbf w} \times  \overline{T}_{>0}  \xrightarrow{\sim}\mathcal S^{>0}_{u}, \quad (g, \overline{t}) \mapsto gtK.
\]
\item The stratification $\mathcal S_{\ge 0} = \sqcup_{u \in Q} \mathcal S^{>0}_{u}$ is a cell decomposition. In particular, the stratum $ \mathcal S^{>0}_{u}$ is non-empty, and $ \mathcal S^{>0}_{u} \cong \mathbb {R}_{>0}^{\ell(w^\bullet) - \ell(u) +r}$.

\item Let $u\in Q$. The closure of $\mathcal S_u^{>0}$ is given by
\[
\overline{\mathcal S_u^{>0}}=\bigsqcup_{u\leq u'}\mathcal S_{u'}^{>0}.
\]
In particular, the totally positive part $\mathcal S_{>0}$ is dense and open in $\mathcal S_{\ge 0}$.

\item Let $u\in Q$. The cell $\mathcal S_{u}^{>0}$ is a connected component of $\mathring{\mathcal S}_u( \mathbb R)$.

\item Let $(v,w)\in P$ and write $\sigma(v,w)=(v',w')$. Given any basis of $\breve{Y}$, we can obtain an isomorphism $ \mathbb{R}^r_{>0}  \xrightarrow{\sim} \overline{T}_{>0}$. Similarly, given any basis of $\breve{X}$, we can obtain an isomorphism $  \overline{T}_{>0}\xrightarrow{\sim} \mathbb{R}^r_{>0} $. The following composition is subtraction-free:
\[
(\mathbb{R}_{>0})^{\ell(w) - \ell(v) +r }   \rightarrow G^{>0}_{\mathbf v_+,\mathbf w} \times  \overline{T}_{>0}  \rightarrow \mathcal S^{>0}_{\phi(v,w)}\rightarrow H^{>0}_{\mathbf v'_+,\mathbf w'} \times  \overline{T}_{>0} \rightarrow (\mathbb{R}_{>0})^{\ell(w) - \ell(v) +r }.
\]
Note that the choice of bases is irrelevant to this claim. 
\item The totally nonnegative symmetric space $\mathcal S_{\ge 0}$ is contractible. 
\end{enumerate}
\end{theorem}

The rest of the paper is devoted to proving the main theorem. We first record a basic example showing that the construction recovers Lusztig’s total positivity for reductive groups.

\begin{example} Let $\tilde{G} = G \times G$.  Let $\tilde{B}^+ = B^+ \times B^-$ and $\tilde{B}^- = B^- \times B^+$.  Let $\varepsilon$ be the involution of $G$ characterized by $
\varepsilon(x_i(a))=x_i(-a)$, $ \varepsilon(y_i(a))=y_i(-a)$, $\varepsilon(t)=t$, for $i\in I$ and $t \in T$. Define the involution $\theta$ on the ambient group $\tilde{G}$ by $\theta((g, h)) = (\varepsilon(h), \varepsilon(g))$. The fixed-point subgroup $\tilde{K}$ is isomorphic to $G$. The symmetric space $\tilde{G}/\tilde{K}$ is also isomorphic to $G$ via the map 
\[
\tilde{G}/\tilde{K} \rightarrow G, \quad (g,h) \mapsto g \varepsilon(h)^{-1}.
\]

Under this isomorphism, we have $(\tilde{G}/\tilde{K})_{\ge 0} \cong G_{\ge 0}$. In particular, we recover various results by Lusztig in \cite{LuTotalPositivity}. For example, the stratification for $(\tilde{G}/\tilde{K})_{\ge 0}$ is identified with the stratification of $G_{\ge 0}$ by double Bruhat cells. 
\end{example}

\begin{remark}
The choice of the Borel subgroups $\tilde{B}^\pm = (B^\pm, B^\mp)$ and the definition of the involution $\theta$ on $\tilde{G}$ are consistent with the rest of the paper and the construction in \cites{BaoSongSymmetricSubgroupSchemesFrobenius, SongSymmetricSubgroupSchemes} from the point of view of quantum symmetric pairs. 
\end{remark}

\subsection{Rank-one calculations}\label{subsec:Rank-One}

For $i\in I_\circ$ and $a,b \in \mathbb {R}_{>0}$, the goal of this subsection is to directly compute $a'$ and $b'$ such that
\begin{equation}\label{eq:rankone}
y_i(a)\alpha_{i}^\vee(b)K=\dot w_{\bullet}^{-1} x_{\tau(i)}(a')\alpha_{\tau(i)}^\vee(b')K.
\end{equation}
One should view this equation as the generalization of the exchange relations in \S\ref{sec:Lutp}. Thanks to the isomorphisms $B^-K/K \cong U_{P^-_{I_\bullet}} \times \overline{T}$ and $B^+K/K \cong U_{P^+_{I_\bullet}} \times \overline{T}$, the solution for $a'$ (if it exists) is unique, while the solution for $b'$ (if it exists) is unique only if we pass to the quotient torus. 
\begin{remark}
The existence of positive solutions, that is, $a', b' \in \mathbb {R}_{>0}$, imposes nontrivial  conditions on the parameters $\overline{\zeta}_i$; see \S\ref{subsec:algebraic-groups}.
\end{remark}

\begin{table}[htbp]
\centering
\caption{Satake diagrams of symmetric pairs of real rank one}
\label{tab:rank-one-satake}
\small
\newcommand{\wnode}[2]{\node[circle,draw,fill=white,inner sep=0pt,minimum size=4.2pt] (#1) at #2 {};}
\newcommand{\bnode}[2]{\node[circle,draw,fill=black,inner sep=0pt,minimum size=4.2pt] (#1) at #2 {};}
\newcommand{\nlabel}[2]{\node[below=2pt] at #1 {\scriptsize $#2$};}
\renewcommand{\arraystretch}{1.45}
\begin{tabular}{|c|c||c|c|}
\hline
AI$_1$ &
\begin{tikzpicture}[baseline=-.55ex,scale=.9]
\wnode{a1}{(0,0)};\nlabel{(0,0)}{1}
\end{tikzpicture}
&
AII$_3$ &
\begin{tikzpicture}[baseline=-.55ex,scale=.9]
\bnode{a1}{(0,0)};\wnode{a2}{(.45,0)};\bnode{a3}{(.9,0)};
\draw (a1)--(a2)--(a3);
\nlabel{(0,0)}{1}\nlabel{(.45,0)}{2}\nlabel{(.9,0)}{3}
\end{tikzpicture}
\\ \hline
AIII$_{11}$ &
\begin{tikzpicture}[baseline=-.55ex,scale=.9]
\wnode{a1}{(0,0)};\wnode{a2}{(.85,0)};
\draw[<->] (a1) to[bend left=45] (a2);
\nlabel{(0,0)}{1}\nlabel{(.85,0)}{2}
\end{tikzpicture}
&
AIV$_n$, $n\geq 2$ &
\begin{tikzpicture}[baseline=-.55ex,scale=.9]
\wnode{a1}{(0,0)};\bnode{a2}{(.45,0)};\bnode{a3}{(2.15,0)};\wnode{a4}{(2.6,0)};
\draw (a1)--(a2);\draw[dashed] (a2)--(a3);\draw (a3)--(a4);
\draw[<->] (a1) to[bend left=45] (a4);
\nlabel{(0,0)}{1}\nlabel{(.45,0)}{2}\nlabel{(2.6,0)}{n}
\end{tikzpicture}
\\ \hline
BII$_n$, $n\geq 2$ &
\begin{tikzpicture}[baseline=-.55ex,scale=.9]
\wnode{a1}{(0,0)};\bnode{a2}{(.45,0)};\bnode{a3}{(2.05,0)};\bnode{a4}{(2.6,0)};
\draw (a1)--(a2);\draw[dashed] (a2)--(a3);\draw[satake double arrow] (a3.east)--(a4.west);
\nlabel{(0,0)}{1}\nlabel{(.45,0)}{2}\node[below=2pt] at (2.05,0) {\scriptsize $n\!-1$};\nlabel{(2.6,0)}{n}
\end{tikzpicture}
&
CII$_{n,1}$, $n\geq 3$ &
\begin{tikzpicture}[baseline=-.55ex,scale=.9]
\bnode{a1}{(0,0)};\wnode{a2}{(.45,0)};\bnode{a3}{(.9,0)};\bnode{a4}{(2.2,0)};\bnode{a5}{(2.75,0)};
\draw (a1)--(a2)--(a3);\draw[dashed] (a3)--(a4);\draw[satake double arrow] (a5.west)--(a4.east);
\nlabel{(0,0)}{1}\nlabel{(.45,0)}{2}\nlabel{(2.75,0)}{n}
\end{tikzpicture}
\\ \hline
DII$_n$, $n\geq 4$ &
\begin{tikzpicture}[baseline=-.55ex,scale=.9]
\wnode{a1}{(0,0)};\bnode{a2}{(.45,0)};\bnode{a3}{(2.0,0)};\bnode{a4}{(2.35,.45)};\bnode{a5}{(2.35,-.45)};
\draw (a1)--(a2);\draw[dashed] (a2)--(a3);\draw (a3)--(a4);\draw (a3)--(a5);
\nlabel{(0,0)}{1}\nlabel{(.45,0)}{2}\node[right=2pt] at (a4) {\scriptsize $n\!-1$};\node[right=2pt] at (a5) {\scriptsize $n$};
\end{tikzpicture}
&
FII &
\begin{tikzpicture}[baseline=-.55ex,scale=.9]
\bnode{a1}{(0,0)};\bnode{a2}{(.45,0)};\bnode{a3}{(.95,0)};\wnode{a4}{(1.4,0)};
\draw (a1)--(a2);\draw[satake double arrow] (a2.east)--(a3.west);\draw (a3)--(a4);
\nlabel{(0,0)}{1}\nlabel{(.45,0)}{2}\nlabel{(.95,0)}{3}\nlabel{(1.4,0)}{4}
\end{tikzpicture}
\\ \hline
\end{tabular}
\end{table}

Let $i\in I_\circ$. We construct the subset $I_i \subset I$ attached to the $\langle\tau\rangle$-orbit of $i$ as follows: we first remove $j \in I_\circ$ such that $j \neq i$ and $j \neq \tau(i)$ from $I$; then we remove any $j \in I_\bullet$ such that $j$ is not in the same connected component as $i$ or $\tau(i)$ in the resulting Dynkin diagram from the first step. The resulting Satake diagram was called real-rank-one in \cite{BaoWangiCanonicalBasis}*{\S3.2 and Definition~3.2} as listed in Table 1. We have the natural decomposition $I_i = I_{i,\circ} \sqcup I_{i,\bullet}$. Let $w_{i, \bullet}$ be the longest element of the Weyl group $W_{I_{i,\bullet}}$.

We then define $L_{I_i}$ as the Levi subgroup of $G$ associated to $I_{i} \subset I$. It follows from the construction that $L_{I_i}$ is $\theta$-stable. Let $K_{I_i} = L_{I_i} \cap K$ be the fixed-point subgroup of $L_{I_i}$. 

We first consider the case when the symmetric pair $(L_{I_i},K_{I_i})$ is of type AI$_1$. In this case, we have, by direct computation,
\begin{align}\label{eq:AI}
y_i(a)\alpha_i^\vee(b)K   &= x_i\left(\frac{ab^4}{1+a^2b^4}\right)\alpha_i^\vee\left(\frac{b}{\sqrt{1+a^2b^4}}\right)K \\
\notag & \stackrel{(\heartsuit)}{=} \dot{w}_\bullet^{-1}  x_i\left(\frac{ab^4}{1+a^2b^4}\right)\alpha_i^\vee\left(\frac{b}{\sqrt{1+a^2b^4}}\right)K.  
\end{align}
The identity $({\heartsuit})$ follows from the construction of $I_i$. 
 
\begin{remark}\label{rem:rankonesubtractionfree}
The appearance of the square root in the formula is superficial. Once we pass to the coordinates of the quotient torus, it will be subtraction-free, without any square root. Upon suitable choices of the cocharacter and the character of $\overline{T}$, the restriction of the birational map $U_{P^-_{I_\bullet}} \times \overline{T} \dashrightarrow  B^-K/K \cap B^+K/K  \dashrightarrow U_{P^+_{I_\bullet}} \times \overline{T}$ becomes 
\begin{align*}
\mathbb{R}_{>0} \times \mathbb{R}_{>0} &\rightarrow \mathbb{R}_{>0} \times \mathbb{R}_{>0}, \\
(s,t)  \mapsto y_i(s)\alpha_i^\vee(\sqrt{t})K   & = x_i\left(\frac{st^2}{1+s^2t^2}\right)\alpha_i^\vee\left(\frac{\sqrt{t}}{\sqrt{1+s^2t^2}}\right)K \\
&\mapsto (\frac{st^2}{1+s^2t^2}, \frac{t}{1+s^2t^2}).
\end{align*}
Note that the choice of the square root is irrelevant, since $\alpha^\vee_i(-1) \in K$ by our assumption on $(L_{I_i},K_{I_i})$.
\end{remark}

Next we consider all other types together. Recall that for type AIV$_{n}$, we shall only consider the case when $n$ is even by the assumption in \S\ref{subsec:algebraic-groups}.

Let \(i\in I_\circ\) belong to a real-rank-one Satake subdiagram which is not of type
\(\mathrm{AI}_1\). Put $\beta_i:=w_{i, \bullet}^{-1}\alpha_{\tau(i)}$. Note that $\beta_i=  w_{ \bullet}^{-1}\alpha_{\tau(i)}$ as well by the construction of $I_i$.  

 \begin{lemma}
The root subgroups \(U_{-\alpha_i}\) and \(U_{\beta_i}\) commute. Equivalently, there is no root of the form \( -r\alpha_i+s\beta_i\) with \(r,s>0\); in particular \(\beta_i-\alpha_i\) is not a root.
\end{lemma}

\begin{proof}
We verify the assertion case by case for the remaining real-rank-one Satake
subdiagrams. To simplify notation, we shall simply assume $(G,K) = (L_{I_i}, K_{I_i})$. We use the standard realizations of the classical root systems.

\smallskip

\noindent
\emph{Type \(\mathrm{AIII}_{11}\).}
Here
\[
        I_\circ=\{1,2\},\qquad \tau(1)=2,\qquad \tau(2)=1,
        \qquad I_\bullet=\varnothing.
\]
Taking \(i=1\), we have \(w_\bullet=1\) and hence
\[
        \beta_1=\alpha_2.
\]
The two simple roots \(\alpha_1\) and \(\alpha_2\) are disconnected in this
rank-one local diagram. Therefore
\[
        \beta_1-\alpha_1=\alpha_2-\alpha_1
\]
is not a root. The case \(i=2\) is identical.

\smallskip

\noindent
\emph{Type \(\mathrm{AIV}\).}
Let the local diagram be of type \(A_n\),
\[
        \circ-\bullet-\cdots-\bullet-\circ,
\]
with
\[
        i=1,\qquad \tau(i)=n,\qquad I_\bullet=\{2,\ldots,n-1\}.
\]
The black Weyl group is the Weyl group of the \(A_{n-2}\)-subsystem generated
by \(\alpha_2,\ldots,\alpha_{n-1}\). Hence
\[
        \beta_1=w_\bullet^{-1}\alpha_n
        =\alpha_2+\alpha_3+\cdots+\alpha_n.
\]
Thus
\[
        \beta_1-\alpha_1
        =
        -\alpha_1+\alpha_2+\cdots+\alpha_n.
\]
In type \(A_n\), every root is of the form
\[
        \pm(\alpha_p+\alpha_{p+1}+\cdots+\alpha_q)
\]
for some \(1\le p\le q\le n\). The expression
\(-\alpha_1+\alpha_2+\cdots+\alpha_n\) has both a negative and positive
coefficient in the basis of simple roots, and therefore is not a root. The case
\(i=n\) is symmetric.

\smallskip

\noindent
\emph{Type \(\mathrm{AII}_3\).}
The local diagram is of type \(A_3\),
\[
        \bullet-\circ-\bullet,
\]
with
\[
        i=2,\qquad I_\bullet=\{1,3\},\qquad \tau(2)=2.
\]
Then
\[
        w_\bullet=s_1s_3,
\]
and therefore
\[
        \beta_2=w_\bullet^{-1}\alpha_2=s_3s_1\alpha_2.
\]
Since
\[
        s_1\alpha_2=\alpha_1+\alpha_2,
\]
we get
\[
        \beta_2=s_3(\alpha_1+\alpha_2)
        =\alpha_1+\alpha_2+\alpha_3.
\]
Hence
\[
        \beta_2-\alpha_2=\alpha_1+\alpha_3.
\]
This is not a root in type \(A_3\), since the positive roots are precisely the
consecutive sums
\[
        \alpha_p+\alpha_{p+1}+\cdots+\alpha_q.
\]
The sum \(\alpha_1+\alpha_3\) is not consecutive. Thus
\(\beta_2-\alpha_2\) is not a root.

\smallskip

\noindent
\emph{Type \(\mathrm{BII}\).}
Use the standard realization of type \(B_n\):
\[
        \alpha_1=e_1-e_2,\qquad
        \alpha_2=e_2-e_3,\qquad \ldots,\qquad
        \alpha_n=e_n.
\]
The local diagram has
\[
        i=1,\qquad I_\bullet=\{2,\ldots,n\},\qquad \tau(1)=1.
\]
The black subsystem is of type \(B_{n-1}\), acting on the span of
\(e_2,\ldots,e_n\). Its longest element sends \(e_2\) to \(-e_2\). Therefore
\[
        \beta_1=w_\bullet^{-1}\alpha_1
        =
        w_\bullet^{-1}(e_1-e_2)
        =
        e_1+e_2.
\]
It follows that
\[
        \beta_1-\alpha_1
        =
        (e_1+e_2)-(e_1-e_2)
        =
        2e_2.
\]
The roots of type \(B_n\) are
\[
        \pm e_p,\qquad \pm e_p\pm e_q \quad (p\neq q).
\]
Thus \(2e_2\) is not a root. Hence \(\beta_1-\alpha_1\) is not a root.

\smallskip

\noindent
\emph{Type \(\mathrm{CII}\).}
Use the standard realization of type \(C_n\):
\[
        \alpha_1=e_1-e_2,\qquad
        \alpha_2=e_2-e_3,\qquad \ldots,\qquad
        \alpha_{n-1}=e_{n-1}-e_n,\qquad
        \alpha_n=2e_n.
\]
The local diagram has
\[
        i=2,\qquad I_\bullet=\{1,3,\ldots,n\},\qquad \tau(2)=2.
\]
The black subsystem is the product of the subsystem generated by \(\alpha_1\)
and the subsystem generated by \(\alpha_3,\ldots,\alpha_n\). Hence
\[
        \beta_2=w_\bullet^{-1}\alpha_2=e_1+e_3.
\]
Therefore
\[
        \beta_2-\alpha_2
        =
        (e_1+e_3)-(e_2-e_3)
        =
        e_1-e_2+2e_3.
\]
The roots of type \(C_n\) are
\[
        \pm 2e_p,\qquad \pm e_p\pm e_q \quad (p\neq q).
\]
The vector \(e_1-e_2+2e_3\) is of neither form. Hence
\(\beta_2-\alpha_2\) is not a root.

\smallskip

\noindent
\emph{Type \(\mathrm{DII}\).}
Use the standard realization of type \(D_n\):
\[
        \alpha_1=e_1-e_2,\qquad
        \alpha_2=e_2-e_3,\qquad \ldots,\qquad
        \alpha_{n-1}=e_{n-1}-e_n,\qquad
        \alpha_n=e_{n-1}+e_n.
\]
The local diagram has
\[
        i=1,\qquad I_\bullet=\{2,\ldots,n\},\qquad \tau(1)=1.
\]
The black subsystem is of type \(D_{n-1}\), acting on the span of
\(e_2,\ldots,e_n\). The relevant black longest element sends \(e_2\) to
\(-e_2\). Hence
\[
        \beta_1=w_\bullet^{-1}\alpha_1=e_1+e_2.
\]
Thus
\[
        \beta_1-\alpha_1
        =
        (e_1+e_2)-(e_1-e_2)
        =
        2e_2.
\]
The roots of type \(D_n\) are
\[
        \pm e_p\pm e_q \quad (p\neq q).
\]
Thus \(2e_2\) is not a root. Hence \(\beta_1-\alpha_1\) is not a root.

\smallskip

\noindent
\emph{Type \(\mathrm{FII}\).}
We use the standard Bourbaki realization of the root system of type \(F_4\):
\[
        \alpha_1=e_2-e_3,\qquad
        \alpha_2=e_3-e_4,\qquad
        \alpha_3=e_4,\qquad
        \alpha_4=\frac{1}{2}(e_1-e_2-e_3-e_4).
\]
The local diagram has
\[
        i=4,\qquad I_\bullet=\{1,2,3\},\qquad \tau(4)=4.
\]
The black subsystem generated by \(\alpha_1,\alpha_2,\alpha_3\) is of type
\(B_3\), acting on the span of \(e_2,e_3,e_4\). Its longest element sends
\(e_j\) to \(-e_j\) for \(j=2,3,4\), and fixes \(e_1\). Hence
\[
        \beta_4=w_\bullet^{-1}\alpha_4
        =
        \frac{1}{2}(e_1+e_2+e_3+e_4).
\]
Therefore
\[
        \beta_4-\alpha_4
        =
        \frac{1}{2}(e_1+e_2+e_3+e_4)
        -
        \frac{1}{2}(e_1-e_2-e_3-e_4)
        =
        e_2+e_3+e_4.
\]
The roots of \(F_4\) in this realization are
\[
        \pm e_p,\qquad
        \pm e_p\pm e_q\quad (p\neq q),
        \qquad
        \frac{1}{2}(\pm e_1\pm e_2\pm e_3\pm e_4).
\]
The vector \(e_2+e_3+e_4\) is not of any of these forms. Hence
\(\beta_4-\alpha_4\) is not a root.

\smallskip

In every non-\(\mathrm{AI}_1\) real-rank-one case, we have shown that
\(\beta_i-\alpha_i\) is not a root. By  \cite{SpringerLinearAlgebraicGroups}*{Proposition~8.2.3}, the
commutator of \(U_{-\alpha_i}\) and \(U_{\beta_i}\) is generated by root subgroups
corresponding to roots of the form
\[
        -r\alpha_i+s\beta_i,\qquad r,s>0.
\]
In the rank-two subsystem generated by \(-\alpha_i\) and \(\beta_i\), the only
possible obstruction is \(\beta_i-\alpha_i\). Since this is not a root in all the
cases above, the two root subgroups commute: $[U_{-\alpha_i},U_{\beta_i}]=1$.
\end{proof}

Therefore, when $(L_{I_i},K_{I_i})$ is of any real-rank-one type other than AI$_1$, we have 
\[
\iota(y_i(a) K) = y_{i}(a) \dot w^{-1}_\bullet x_{\tau(i)}(a) \dot w_\bullet = \iota(\dot w^{-1}_\bullet x_{\tau(i)}(a)K).
\]
We conclude that 
\[
y_i(a) K=\dot w^{-1}_\bullet x_{\tau(i)}(a)K = \dot w^{-1}_{i,\bullet} x_{\tau(i)}(a)K.
\]
It then follows immediately that 
\begin{equation}\label{eq:otherrankone}
y_i(a) \alpha_i^\vee(b)K=\dot w^{-1}_\bullet x_{\tau(i)}(a b^{2 + \langle \alpha_i^\vee, w_\bullet \alpha_{\tau(i)} \rangle }) \alpha^\vee_{\tau(i)}(b) K.
\end{equation}

\subsection{The special regular function}

The goal of this subsection is to prove the first and most crucial  assertion of Theorem~\ref{thm:main-positive-strata}.

\begin{proposition}\label{prop:positive-borels}
One has
\[
B^-_{>0}K/K=B^+_{>0}K/K.
\]
Moreover,
\[
\mathcal S_{\geq0}=\overline{B^+_{>0}K/K}=\overline{B^-_{>0}K/K}.
\]
\end{proposition}

\begin{proof}We first show $B^-_{>0}K/K\subset B^+_{>0}K/K$. Fix a reduced expression $\mathbf w_\bullet=s_{j_1}\cdots s_{j_m}$ of $w_\bullet$. For a tuple $\mathbf b=(b_1,\ldots,b_m)$ with all entries positive, write
\[
y_{\mathbf w_\bullet}(\mathbf b)=y_{j_1}(b_1)\cdots y_{j_m}(b_m),\qquad x_{\mathbf w_\bullet}(\mathbf b)=x_{j_1}(b_1)\cdots x_{j_m}(b_m).
\]
Choose a (not necessarily reduced) expression
\[
\mathbf w=(\mathbf w_\bullet, s_{i_1}, \mathbf w_\bullet, s_{i_2}, \cdots , \mathbf w_\bullet, s_{i_n}),\qquad i_r\in I_\circ,
\]
such that the Demazure product is $w_0$. By \S\ref{sec:Lutp} (a), we can write an arbitrary element in $B^-_{>0}K/K$ as 
\[
y_{\mathbf w_\bullet}(\mathbf b^{(1)})y_{i_1}(a_1)y_{\mathbf w_\bullet}(\mathbf b^{(2)})y_{i_2}(a_2)\cdots y_{\mathbf w_\bullet}(\mathbf b^{(n)})y_{i_n}(a_n)tK \in B^-_{>0}K,
\]
with $\mathbf b^{(r)}=(b^{(r)}_1,\ldots,b^{(r)}_m) \in \mathbb{R}_{>0}^{\ell(w_\bullet)}$, $a_r \in \mathbb {R}_{>0}$, and $t \in T_{>0}$. 

Then by the rank-one identities  in \S\ref{subsec:Rank-One}, we obtain that 
\begin{align*}
&y_{\mathbf w_\bullet}(\mathbf b^{(1)})y_{i_1}(a_1)\cdots y_{\mathbf w_\bullet}(\mathbf b^{(n)})y_{i_n}(a_n)tK  \\
=& y_{\mathbf w_\bullet}(\mathbf b^{(1)})y_{i_1}(a_1)\cdots y_{\mathbf w_\bullet}(\mathbf b^{(n)}) \dot{w}_\bullet^{-1} x_{\tau(i_n)}(a_n') t' K, \quad \text{ for } a_n' \in \mathbb{R}_{>0} \text{ and } t \in T_{>0}.
\end{align*}
By direct computation, we have $y_{j_m}(b^{(n)}_m) \dot{s}_{j_m}^{-1} = x_{j_m}((b^{(n)}_m)^{-1}) y_{j_m}(- b^{(n)}_m) \alpha_{j_m}^\vee((b^{(n)}_m)^{-1})$.  The root subgroups $U^-_{w_\bullet s_{j_m} (\alpha_{j_m})}$ and $U^+_{\alpha_{\tau(i_n)}}$ commute by \cite{SpringerLinearAlgebraicGroups}*{Proposition~8.2.3}. Since
\[
U^-_{w_\bullet s_{j_m}(\alpha_{j_m})} \subset [L_{I_\bullet},L_{I_\bullet}] \subset K,
\]
we have
\begin{align}
 \label{eq:xw}  &y_{j_m}(b^{(n)}_m) \dot{w}_\bullet^{-1} x_{\tau(i_n)}(a'_n) t'K \\
 \notag = &x_{j_m}((b^{(n)}_m)^{-1}) \dot{s}_{j_m} \dot{w}_\bullet^{-1} x_{\tau(i_n)}( (b^{(n)}_m)^{-\langle w_\bullet s_{j_m}(\alpha_{j_m}^\vee), \alpha_{\tau(i_n)} \rangle} a'_n) t'K.
\end{align}

We can apply the exchange relations in  \S\ref{sec:Lutp}, and continue the process for the remaining factors in $ y_{\mathbf w_\bullet}(\mathbf b^{(n)})$, we obtain 
\begin{align*}
&y_{\mathbf w_\bullet}(\mathbf b^{(1)})y_{i_1}(a_1)\cdots y_{\mathbf w_\bullet}(\mathbf b^{(n)}) \dot{w}_\bullet^{-1} x_{\tau(i_n)}(a_n') t' K \\
= &y_{\mathbf w_\bullet}(\mathbf b^{(1)})y_{i_1}(a_1) \cdots  y_{\mathbf w_\bullet}(\mathbf b^{(n-1)}) y_{i_{n-1}}(a_{n-1})x_{\mathbf w_\bullet}( (\mathbf b')^{(n)})  x_{\tau(i_n)}(a_n'') t'' K,
\end{align*}
where $(\mathbf b')^{(n)} \in \mathbb{R}_{>0}^{\ell(w_\bullet)}$, $a_{n''} \in \mathbb{R}_{>0}$, and $t'' \in T_{>0}$.

Then by identities in \S\ref{sec:Lutp}, we further have 
\begin{align*}
&y_{\mathbf w_\bullet}(\mathbf b^{(1)})y_{i_1}(a_1) \cdots  y_{\mathbf w_\bullet}(\mathbf b^{(n-1)}) y_{i_{n-1}}(a_{n-1})x_{\mathbf w_\bullet}( (\mathbf b')^{(n)})  x_{\tau(i_n)}(a_n'') t'' K  \\
= &x_{\mathbf w_\bullet}( (\mathbf b'')^{(n)})  x_{\tau(i_n)}(a_n''') y_{\mathbf w_\bullet}((\mathbf b'')^{(1)})y_{i_1}(a''_1) \cdots  y_{\mathbf w_\bullet}((\mathbf b'')^{(n-1)}) y_{i_{n-1}}(a''_{n-1}) t'''K.
\end{align*}
Here all coordinates are in $\mathbb{R}_{>0}$ and  $t''' \in T_{>0}$.  Therefore, by repeating the process, we have 
\begin{align*}
&y_{\mathbf w_\bullet}(\mathbf b^{(1)})y_{i_1}(a_1) \cdots y_{\mathbf w_\bullet}(\mathbf b^{(n)})y_{i_n}(a_n)tK \\
= 
&x_{\mathbf w_\bullet}(\mathbf c^{(n)})x_{\tau(i_n)}(c_n) \cdots x_{\mathbf w_\bullet}(\mathbf c^{(1)})x_{\tau(i_1)}(c_1 ) s K.
\end{align*}
Here $\mathbf c^{(r)} \in \mathbb{R}_{>0}^{\ell(w_\bullet)}$, $c_i \in \mathbb{R}_{>0}$, and $s \in T_{>0}$.  Hence $B^-_{>0}K/K\subset B^+_{>0}K/K$. The opposite inclusion is obtained symmetrically. Therefore $B^-_{>0}K/K=B^+_{>0}K/K$.

Now we have 
\[
G_{>0}K/K = B^-_{>0} U^+_{>0}K/K \subset B^-_{>0} B^-_{>0} K/K \subset  B^-_{>0}K/K =B^+_{>0}K/K.
\]
It is clear from the definition that $B^-_{>0}K/K=B^+_{>0}K/K\subset(G/K)_{\geq0}$. Taking Hausdorff closures gives the desired equality.
\end{proof}

Let $\lambda$ be a regular dominant weight, and let $V(\lambda)$ be the irreducible highest-weight $G$-module of highest weight $\lambda$. Fix a  highest weight vector $\eta_\lambda \in V(\lambda)$ and a  lowest weight vector $\xi_{-\lambda} \in V(-w_0\lambda)$. For each $w\in W$, we fix nonzero extremal weight vectors
\[
\eta_{w\lambda} = \dot{w} \cdot \eta_\lambda \in V(\lambda)_{w\lambda},\qquad \xi_{-w\lambda} = \dot{w} \cdot\xi_{-\lambda} \in V(-w_0\lambda)_{-w\lambda}.
\]
We use the generalized-minor notation of Marsh--Rietsch \cite{MarshRietschParametrizationsFlagVarieties}*{\S 7}. Thus $\Delta^{u\lambda}_{v\lambda}(g)$
denotes the coefficient of $\eta_{u\lambda}$ in $g\eta_{v\lambda}$. Similarly, we write $\nabla^{-u\lambda}_{-v\lambda}(g)$ for the coefficient of $\xi_{-u\lambda}$ in $g\xi_{-v\lambda}$. For $w\in W$, define the regular function $\Pi^\lambda_w:G\to\mathbb A^1$ by
\[
\Pi^\lambda_w(g)=\Delta^{w\lambda}_{w\lambda}(g)\,\nabla^{-w\lambda}_{-w\lambda}(g).
\]
Via the embedding $G/K \rightarrow G$, the function $\Pi^\lambda_w$ induces a function on $G/K$, still denoted by $\Pi^\lambda_w$:
\[
\Pi^\lambda_w(gK)=\Pi^\lambda_w(g\theta(g)^{-1}).
\]

Now let us give a new proof of an essential part of Lusztig's theorem \cite{LuTotalPositivity}*{Theorem~4.3} without using his theory of canonical bases. This theorem can be seen as a special case of Theorem~\ref{thm::(1)}.

\begin{thm}\label{thm:Lu4.3}
$\overline{G_{\ge 0}} \subset (U^- T U^+) \cap (U^+TU^-)$. 
\end{thm}
\begin{proof}We show $\overline{G_{\ge 0}} \subset U^- T U^+$. The inclusion $\overline{G_{\ge 0}} \subset U^+ T U^-$ can be proved similarly. 

Following Lusztig, it suffices to show $\Pi^\lambda_e(s) \geq 1$ for any $ s \in G_{>0}$, and $\Pi^\lambda_e(g) = 0$ for any $g \in B^+\dot s_iB^-$ or $g \in  B^-\dot s_iB^+$. 

Let $\mathbf{w}_0 = s_{i_1} \dots s_{i_n}$ be a reduced expression of $w_0$. We can then write an arbitrary element $s \in G_{>0}$ as the product ($a_i,b_i>0$ and $t\in T_{>0}$)
\[
y_{i_1}(a_1) y_{i_2}(a_2) \cdots y_{i_n}(a_n) x_{i_n}(b_n)  \cdots x_{i_1}(b_1)t.
\]

{\it (a) We claim $\Pi^\lambda_e(s) \ge 1$. } 

Let $s' = y_{i_1}(a_1) y_{i_2}(a_2) \cdots y_{i_n}(a_n) x_{i_n}(b_n)  \cdots x_{i_1}(b_1)$. It is straightforward that $\Pi^\lambda_e(s) = \Pi^\lambda_e(s') = \nabla^{-\lambda}_{-\lambda}(s')$. 

We compute 
\begin{align*}
&\nabla^{-\lambda}_{-\lambda}(y_{i_1}(a_1) y_{i_2}(a_2) \cdots y_{i_n}(a_n) x_{i_n}(b_n)  \cdots x_{i_1}(b_1)  ) \\
= & \nabla^{-\lambda}_{-\lambda}(y_{i_1}(a_1) y_{i_2}(a_2) \cdots y_{i_{n-1}}(a_{n-1}) x_{i_n}(\frac{b_n}{a_nb_n +1}) \alpha_{i_n}^\vee(\frac{1}{a_nb_n+1}) y_{i_n}(\frac{a_n}{a_nb_n +1}) \\
&\cdot x_{i_{n-1}}(b_{n-1}) \cdots x_{i_1}(b_1)) \\
= & \nabla^{-\lambda}_{-\lambda}(y_{i_1}(a_1) y_{i_2}(a_2) \cdots y_{i_{n-1}}(a_{n-1}) x_{i_n}(\frac{b_n}{a_nb_n +1})  y_{i_n}(c_n) \\
&\cdot x_{i_{n-1}}(c_{n-1}) \cdots x_{i_1}(c_1)) \alpha_{i_n}^\vee(\frac{1}{a_nb_n+1}).
\end{align*}

Note that $\alpha_{i_n}^\vee(\frac{1}{a_nb_n+1}) \xi_{-\lambda} = ({a_nb_n+1})^{\langle \alpha_{i_n}^\vee, \lambda\rangle } \xi_{-\lambda}$. We know $\langle \alpha_{i_n}^\vee, \lambda\rangle \ge 1$, since $\lambda $ is regular dominant. We conclude that $({a_nb_n+1})^{\langle \alpha_{i_n}^\vee, \lambda\rangle } \ge 1$, since $a_n, b_n > 0$. Therefore
\begin{align*}
&\nabla^{-\lambda}_{-\lambda}(y_{i_1}(a_1) y_{i_2}(a_2) \cdots y_{i_n}(a_n) x_{i_n}(b_n)  \cdots x_{i_1}(b_1)  ) \\
\ge & \nabla^{-\lambda}_{-\lambda}(y_{i_1}(a_1) y_{i_2}(a_2) \cdots y_{i_{n-1}}(a_{n-1}) x_{i_n}(\frac{b_n}{a_nb_n +1})  y_{i_n}(c_n) \\
&\cdot x_{i_{n-1}}(c_{n-1}) \cdots x_{i_1}(c_1)) \\
\ge & \nabla^{-\lambda}_{-\lambda} (x_{i_n}(b'_n)  \cdots x_{i_1}(b'_1) y_{i_1}(a'_1) y_{i_2}(a_2) \cdots y_{i_n}(a'_n)) =1.
\end{align*}
This shows $(a)$. 

{\it (b) We claim $\Pi^\lambda_e(g) = 0 $ for any simple reflection $s_i$, and for $g \in B^+\dot s_iB^-$ or $g \in B^-\dot s_iB^+$.} 

If $g \in B^-\dot s_iB^+$, then we have $\Delta^{\lambda}_{\lambda}(g) = 0$. If $g \in B^+\dot s_iB^-$, then we have $\nabla^{-\lambda}_{-\lambda}(g) = 0$. This shows $(b)$. 

This completes the proof.
\end{proof}

We now prove the generalization to symmetric spaces. 

\begin{theorem}\label{thm::(1)}
The totally nonnegative symmetric space satisfies
\[
\mathcal S_{\geq0}\subset(B^+K/K)\cap(B^-K/K).
\]
\end{theorem}

\begin{proof}
The theorem follows from the following two claims. 

{\it (a) For $g\in G_{>0}$, we have $\Pi^\lambda_{w_\bullet}(gK)\geq1$.} 

{\it (b) We have $\Pi^\lambda_{w_\bullet}(gK)=0$ for $gK\notin B^+K/K$ or $gK\notin B^-K/K$.}

We first show claim (a). The idea of the proof is the same as the proof of Theorem~\ref{thm:Lu4.3} using results from \S\ref{subsec:Rank-One}. 

By Proposition~\ref{prop:positive-borels}, we know $gK=g'K$ for some $g'\in B^-_{>0}$. Let $w_0 = s_{i_1}\cdots s_{i_n}$ be reduced. Then we can write
\[
g'=tb, \quad \text{ with } b=y_{i_1}(a_1)\cdots y_{i_n}(a_n), a_r>0, t\in T_{>0}.
\]
Since that $\Pi^\lambda_{w_\bullet}$ is invariant under left translation by $T$, we have $\Pi^\lambda_{w_\bullet}(g'K) = \Pi^\lambda_{w_\bullet}(bK)$. Note that $b \theta(b)^{-1} \in U^- \dot{w}_\bullet U^+ \dot{w}_\bullet^{-1}$. Therefore $\Delta^{w_\bullet\lambda}_{w_\bullet\lambda} (b \theta(b)^{-1}) =1$. Hence $\Pi^\lambda_{w_\bullet}(g'K) = \Pi^\lambda_{w_\bullet}(bK) = \nabla^{-w_\bullet\lambda}_{-w_\bullet\lambda} (b \theta(b)^{-1})$. 

By Proposition~\ref{prop:positive-borels} again, we can find
\[
c=x_{i_1}(d_1)x_{i_2}(d_2)\cdots x_{i_n}(d_n)s\in B^+_{>0}, \text{ with $s \in T_{>0}$ such that $bK=cK$ }.
\]
Then
\[
\nabla^{-w_\bullet\lambda}_{-w_\bullet\lambda}(b\theta(b)^{-1}) = \nabla^{-w_\bullet\lambda}_{-w_\bullet\lambda}(c\theta(c)^{-1})=\langle\xi_{-w_\bullet\lambda},c\theta(c)^{-1}\xi_{-w_\bullet\lambda}\rangle=\langle\xi_{-w_\bullet\lambda},s\theta(s)^{-1}\xi_{-w_\bullet\lambda}\rangle.
\]
Now the computation for the element $s \theta{(s)}^{-1}$ is entirely similar to the computation in the proof of Theorem~\ref{thm:Lu4.3}. This follows by repeatedly applying \S\ref{subsec:Rank-One} (a) (and (b) trivially) and \eqref{eq:Luxy}.  This proves claim (a). 

We next show claim (b). 

Assume $gK$ is not in $B^-K/K$. We show $\Pi^\lambda_{w_\bullet}(gK)=0$ by showing that $\Delta^{w_\bullet\lambda}_{w_\bullet\lambda}(g\theta(g)^{-1})=0$. The other statement is similar.

Recall  $B^-K/K$ is precisely the unique open $B^-$-orbit. Let $B^-zK/K$ be a codimension-one $B^-$-orbit for some $z\in G$ such that the closure contains $gK$. By Springer \cite{SpringerAlgebraicGroupsWithInvolutions}*{\S6} (see also \cite{BaoSongSymmetricSubgroupSchemesFrobenius}*{\S2.4}), the double coset  $B^-z \theta(z) \theta(B^-)$ is either $B^-\dot s_i \theta(\dot s_i^{-1})\theta(B^-)$ or $B^-\dot s_i \theta(B^-)$ for some $i \in I_{\circ}$. In any case, we have
\[
\iota(B^-z)\subset B^-\dot{w}\dot{w}_\bullet B^+\dot{w}_{\bullet}^{-1},\qquad \text{with } w^\bullet \le  ww^\bullet.
\]
It follows that $\Delta^{w_\bullet\lambda}_{w_\bullet\lambda}(B^-zK/K)=0$, and hence $\Delta^{w_\bullet\lambda}_{w_\bullet\lambda}(\overline{B^-zK/K})=0$. Therefore $\Delta^{w_\bullet\lambda}_{w_\bullet\lambda}(gK) = 0$.

This completes the proof.
\end{proof}

\subsection{Proof of the main theorem}

In this subsection we prove the remaining assertions of Theorem~\ref{thm:main-positive-strata}.

\begin{proof}
\begin{itemize}[leftmargin=*]
\item Part (1) is proved in Theorem~\ref{thm::(1)}. 
\item 
We show Part (2) now. Recall the isomorphism $f: B^-K/K \cong U_{P_{I_\bullet}^-} \times \overline{T}$. We further denote by $f_1:  B^-K/K  \rightarrow  U_{P_{I_\bullet}^-}$ and $f_2: B^-K/K  \rightarrow  \overline{T}$ the projections to each component. Hence by Part (1), we have (in the Hausdorff topology) 
\[
\overline{G_{>0} K/K} = f^{-1} ( \overline{f_1 (G_{>0} K/K)} \times \overline{f_2 (G_{>0} K/K)}).
\]
By Lemma~\ref{lem:Tbar}, we have $\overline{f_2 (G_{>0} K/K)} = \overline{T}_{>0}$. We can also determine $\overline{f_1 (G_{>0} K/K)}$ by \S\ref{sec:partial-flag-tnn} thanks to the commutative diagram 
\[
\begin{tikzcd}
U^- 
  \arrow[r]
  \arrow[d]
&
U^- P^+_{I_\bullet}/P^+_{I_\bullet}
  \arrow[d]
\\
U^-K/K
  \arrow[r]
&
U_{P^-_{I_\bullet}}.
\end{tikzcd}
\] 
Therefore we have
\[
\mathcal S_{\geq0}=\bigsqcup_{(v,w)\in P}G^{>0}_{\mathbf v_+,\mathbf w}  {T}_{>0}K/K.
\]
Thanks to the isomorphism in \S\ref{sec:partial-flag-tnn} (c), it suffices to show
\[
G^{>0}_{\mathbf v_+,\mathbf w} {T}_{>0}K/K\subset \mathring{\mathcal S}_{\phi(v,w)}.
\]
Recall $G^{>0}_{\mathbf v_+,\mathbf w}\subset U^-\dot{v} \cap B^+\dot{w}B^+$ from \S\ref{sec:TPFlag}. By direct computation, we have
\[
\iota(G^{>0}_{\mathbf v_+,\mathbf w}T_{>0}K/K)
\subset U^-\dot v\theta(B^+\dot w^{-1}B^+).
\]
Since $\theta(B^+)=\dot w^\bullet B^+\dot w^{\bullet,-1}$ and
\[
\ell\bigl(v\theta(w^{-1})w^\bullet\bigr)
=\ell(vw^\bullet)-\ell(w),
\]
the usual length-decreasing multiplication rule for opposite Bruhat cells gives
\[
\iota(G^{>0}_{\mathbf v_+,\mathbf w}T_{>0}K/K)
\subset B^-\dot v\theta(\dot w^{-1})\theta(B^+).
\]
Applying the anti-involution $x\mapsto\theta(x^{-1})$, which fixes $\iota(G/K)$ pointwise, and using $\theta(v)=v$, we obtain
\[
\iota(G^{>0}_{\mathbf v_+,\mathbf w}T_{>0}K/K)
\subset B^+\dot w\dot v^{-1}\theta(B^-)
=B^+\dot{\phi(v,w)}B^-\dot w^\bullet.
\]
This proves item (2).

\item For Part (3), apply the Chevalley involution $\omega$ to Part (2). It preserves $\mathcal S_{\geq0}$ and $\overline T_{>0}$, commutes with $\theta$, and sends $G^{>0}_{\mathbf v_+,\mathbf w}$ to $H^{>0}_{\mathbf v_+,\mathbf w}$. By \eqref{eq:omega-strata} and \eqref{eq:sigma-definition},
\[
\omega(\mathcal S_{\phi(v,w)}^{>0})
=\mathcal S_{\rho(\phi(v,w))}^{>0}
=\mathcal S_{\phi(\sigma(v,w))}^{>0}.
\]
This proves Part (3). Since $\sigma^2=\mathrm{id}$, the two parametrizations paired by $\sigma$ have the same image.
\item Part (4) follows from either Part (2) or Part (3).
\item Part (5) follows from \S\ref{sec:partial-flag-tnn}, Proposition~\ref{prop:poset-isomorphism} and the isomorphism $B^-K/K \cong U_{P_{I_\bullet}^-} \times \overline{T} \cong (U^-P^+_{I_\bullet} / P^+_{I_\bullet})\times\overline T$.
\item 
We prove Part (6) now. It follows from Part (2) that $\mathcal{S}_{u}^{>0}$ is connected.  By Lemma~\ref{le:Su}, the subvariety $\mathring{\mathcal  S}_u$ is smooth of dimension $\ell(w^\bullet) - \ell(u) + r$. Hence $\mathring{\mathcal  S}_u (\mathbb {R})$ is smooth of real dimension $\ell(w^\bullet) - \ell(u) + r$. Then by Brouwer's theorem of ``invariance of domain", we see that the parameterization $G^{>0}_{\mathbf v_+,\mathbf w}\times\overline{T}_{>0} \rightarrow \mathring{\mathcal  S}_u(\mathbb R)$ is an open map. Then by Part (5), we know the image is closed. Therefore it is a connected component in the real locus. 

\item We show Part (7). The essential point is already visible in the rank-one exchange relations in \S\ref{subsec:Rank-One} and the exchange relations in \S\ref{sec:Lutp}; see Remark~\ref{rem:rankonesubtractionfree}. It follows immediately that the coordinate change for the case $v=e$ is subtraction-free.

Let $(v,w) \in P$ with $v \neq e$, and write $\sigma(v,w)=(v',w')$. Fix a reduced expression $\mathbf w$ of $w$. Let $\mathbf v_+$ determine the positive expression of $v$ inside $\mathbf w$. Similarly define $\mathbf w'$ and $\mathbf v'_+$. The following bijection is subtraction-free by \cite{BaoHe2021Semifields}*{Proposition~6.2} and the isomorphism $B^-K/K \cong U_{P_{I_\bullet}^-} \times \overline{T} \cong (U^-P^+_{I_\bullet} / P^+_{I_\bullet})\times\overline T$:
\begin{align*}
(\mathbb{R}_{>0})^{\ell(w)+r}
&\longrightarrow
U^+_{v^{-1},>0}\times G^{>0}_{\mathbf v_+,\mathbf w}\times\overline {T}_{>0}
\longrightarrow \mathcal S^{>0}_{\phi(e,w)} \\
&\longrightarrow
G^{>0}_{\mathbf e,\mathbf w}\times\overline{T}_{>0}
\longrightarrow
(\mathbb{R}_{>0})^{\ell(w)+r}.
\end{align*}
Similarly the following bijection is also subtraction-free:
\begin{align*}
(\mathbb{R}_{>0})^{\ell(w')+r}
&\longrightarrow
U^-_{(v')^{-1},>0}\times H^{>0}_{\mathbf v'_+,\mathbf w'}\times\overline{T}_{>0}
\longrightarrow \mathcal S^{>0}_{\phi(\sigma(e,w'))} \\
&\longrightarrow
H^{>0}_{\mathbf e,\mathbf w'}\times\overline{T}_{>0}
\longrightarrow
(\mathbb{R}_{>0})^{\ell(w')+r}.
\end{align*}
By \cite{GKL}*{Lemma~5.9}, taking the coordinates in $U^+_{v^{-1},>0}$ to the boundary gives the desired subtraction-free transition map on the cell $\mathcal S^{>0}_{\phi(v,w)}$. 

\item We prove Part (8) now. By Part (2), we have 
\[
 \mathcal{S}_{\ge 0} \cong \left((U^-P^+_{I_\bullet}/P^+_{I_\bullet})\cap(\mathcal P_{I_\bullet})_{\geq0} \right)\times \overline{T}_{>0}.
\]
It follows from Lemma~\ref{lem:contractible} that $\mathcal{S}_{\ge 0}$ is contractible.
\qedhere
\end{itemize}
\end{proof}


\begin{thebibliography}{10}

\bib{BaoHe2021}{article}{
  author={Bao, Huanchen},
  author={He, Xuhua},
  title={A Birkhoff--Bruhat atlas for partial flag varieties},
  journal={Indag. Math. (N.S.)},
  volume={32},
  date={2021},
  number={5},
  pages={1152--1173},
}

\bib{BaoHe2021Semifields}{article}{
  author={Bao, Huanchen},
  author={He, Xuhua},
  title={Flag manifolds over semifields},
  journal={Algebra Number Theory},
  volume={15},
  date={2021},
  number={8},
  pages={2037--2069},
  doi={10.2140/ant.2021.15.2037},
}

\bib{BaoHe2022TwistedProduct}{article}{
  author={Bao, Huanchen},
  author={He, Xuhua},
  title={Total positivity in twisted product of flag varieties},
  eprint={arXiv:2211.11168},
  date={2022},
}
\bib{BaoHe2024Product}{article}{
  author={Bao, Huanchen},
  author={He, Xuhua},
  title={Product structure and regularity theorem for totally nonnegative flag varieties},
  journal={Invent. Math.},
  volume={237},
  date={2024},
  number={1},
  pages={1--47},
  doi={10.1007/s00222-024-01256-2},
}
\bib{BaoSongCoordinateRingsSymmetricSpaces}{article}{
  author={Bao, Huanchen},
  author={Song, Jinfeng},
  title={Coordinate rings on symmetric spaces},
  eprint={arXiv:2402.08258},
  date={2024},
}


\bib{BaoSongSymmetricSubgroupSchemesFrobenius}{article}{
  author={Bao, Huanchen},
  author={Song, Jinfeng},
  title={Symmetric subgroup schemes, Frobenius splittings, and quantum symmetric pairs},
  eprint={arXiv:2212.13426},
  date={2022},
}

\bib{BaoWangiCanonicalBasis}{article}{
  author={Bao, Huanchen},
  author={Wang, Weiqiang},
  title={Canonical bases arising from quantum symmetric pairs},
  eprint={arXiv:1610.09271},
  date={2018},
}



\bib{CLT10}{article}{
  author={Chan, Kei Yuen},
  author={Lu, Jiang-Hua},
  author={To, Simon Kai-Ming},
  title={On intersections of conjugacy classes and Bruhat cells},
  journal={Transform. Groups},
  volume={15},
  date={2010},
  number={2},
  pages={243--260},
  issn={1083-4362},
  review={\MR{2657442}},
  doi={10.1007/s00031-010-9084-7},
}

\bib{DP}{article}{
  author={De Concini, Corrado},
  author={Procesi, Claudio},
  title={Complete symmetric varieties},
  conference={
    title={Invariant theory},
    address={Montecatini},
    date={1982},
  },
  book={
    series={Lecture Notes in Math.},
    volume={996},
    publisher={Springer},
    place={Berlin},
  },
  date={1983},
  pages={1--44},
}

\bib{EvensLu2006}{article}{
  author={Evens, Sam},
  author={Lu, Jiang-Hua},
  title={On the variety of {L}agrangian subalgebras. {II}},
  journal={Ann. Sci. \'Ecole Norm. Sup. (4)},
  volume={39},
  date={2006},
  number={2},
  pages={347--379},
}

\bib{GKL}{article}{
  author={Galashin, Pavel},
  author={Karp, Steven N.},
  author={Lam, Thomas},
  title={Regularity theorem for totally nonnegative flag varieties},
  journal={J. Amer. Math. Soc.},
  volume={35},
  date={2022},
  number={2},
  pages={513--579},
  doi={10.1090/jams/983},
}

\bib{Lu2014TLeaves}{article}{
  author={Lu, Jiang-Hua},
  title={On the {$T$}-leaves and the ranks of a {P}oisson structure on twisted conjugacy classes},
  journal={Indag. Math. (N.S.)},
  volume={25},
  date={2014},
  number={5},
  pages={1102--1121},
  doi={10.1016/j.indag.2014.07.011},
}



\bib{LuTotalPositivity}{article}{
  author={Lusztig, George},
  title={Total positivity in reductive groups},
  journal={Lie theory and geometry},
  pages={531--568},
  publisher={Birkh\"auser},
  date={1994},
}
\bib{Lusztig1998}{article}{
  author={Lusztig, George},
  title={Introduction to total positivity},
  booktitle={Positivity in {L}ie theory: open problems},
  series={de Gruyter Exp. Math.},
  volume={26},
  publisher={de Gruyter},
  place={Berlin},
  date={1998},
  pages={133--145},
}
\bib{LusztigTotalPositivitySymmetricSpaces}{article}{
  author={Lusztig, George},
  title={Total positivity in symmetric spaces},
  journal={Represent. Theory},
  volume={26},
  date={2022},
  pages={1025--1046},
  eprint={arXiv:2107.13447},
  doi={10.1090/ert/628},
}

\bib{MarshRietschParametrizationsFlagVarieties}{article}{
  author={Marsh, R. J.},
  author={Rietsch, K.},
  title={Parametrizations of flag varieties},
  eprint={arXiv:math/0307017},
  date={2004},
}

\bib{Ri92}{article}{
  author={Richardson, R. W.},
  title={Intersections of double cosets in algebraic groups},
  journal={Indag. Math. (N.S.)},
  volume={3},
  date={1992},
  number={1},
  pages={69--77},
  issn={0019-3577},
  review={\MR{1157520}},
  doi={10.1016/0019-3577(92)90028-J},
}

\bib{RichardsonSpringerBruhatOrderSymmetricVarieties}{article}{
  author={Richardson, R. W.},
  author={Springer, T. A.},
  title={The Bruhat order on symmetric varieties},
  journal={Geom. Dedicata},
  volume={35},
  date={1990},
  pages={389--436},
}

\bib{RietschClosureRelationsPartialFlag}{article}{
  author={Rietsch, K.},
  title={Closure relations for totally nonnegative cells in $G/P$},
  eprint={arXiv:math/0509137},
  date={2005},
}

\bib{SongSymmetricSubgroupSchemes}{article}{
  author={Song, Jinfeng},
  title={Symmetric subgroup schemes},
  eprint={arXiv:2512.01398},
  date={2025},
}

\bib{SpringerLinearAlgebraicGroups}{book}{
  author={Springer, T. A.},
  title={Linear algebraic groups},
  edition={2},
  publisher={Birkh\"auser},
  date={1998},
}

\bib{SpringerAlgebraicGroupsWithInvolutions}{article}{
  author={Springer, T. A.},
  title={Some results on algebraic groups with involutions},
  booktitle={Algebraic groups and related topics},
  series={Advanced Studies in Pure Mathematics},
  volume={6},
  pages={525--543},
  publisher={North-Holland},
  date={1985},
}

\end{thebibliography}
\end{document}